\documentclass[final,leqno]{siamltex}

\usepackage[cmex10]{amsmath}
\usepackage{amssymb}
\usepackage[noadjust]{cite}
\usepackage{graphicx}

\usepackage{enumitem}

\newcommand\defproof[1]{
\def\proof{\par{\it Proof of #1}. \ignorespaces} %
\def\endproof{\vbox{\hrule height0.6pt\hbox{%
   \vrule height1.3ex width0.6pt\hskip0.8ex
   \vrule width0.6pt}\hrule height0.6pt
  }}
}
\newcommand\undefproof{\def\proof{\par{\it Proof}. \ignorespaces}}

\newcommand{\norm}[1]{\lVert #1 \rVert}
\newcommand{\Norm}[1]{\left\lVert #1 \right\rVert}
\newcommand{\onenorm}[1]{\norm{#1}_1}
\newcommand{\oneNorm}[1]{\Norm{#1}_1}
\newcommand{\twonorm}[1]{\norm{#1}_2}

\newcommand{\nnnorm}[1]{{\left\vert\kern-0.25ex\left\vert\kern-0.25ex\left\vert #1 \right\vert\kern-0.25ex\right\vert\kern-0.25ex\right\vert}}
\newcommand{\abs}[1]{\lvert #1 \rvert}
\newcommand{\Abs}[1]{\left\lvert #1 \right\rvert}
\DeclareMathOperator{\spn}{span}

\newcommand{\titlestring}{%
STABILITY ANALYSIS OF POSITIVE SEMI-MARKOVIAN JUMP LINEAR SYSTEMS WITH STATE RESETS}

\newcommand{\mykeywords}{
Semi-Markovian jump linear systems, 
mean stability, 
positive systems, 
Markovian renewal processes}

\title{\titlestring}
\author{Masaki~Ogura and Clyde~F.~Martin%
\thanks{M.~Ogura and C.~F.~Martin are with the Department
of Mathematics and Statistics, Texas Tech University, Lubbock, TX, 79409 USA
(e-mail: {\texttt{{masaki.ogura@ttu.edu}}}, {\texttt{{clyde.f.martin@ttu.edu}}}).}}

\newcommand\hyperrefopt{bookmarks=true,bookmarksnumbered=true,
pdfpagemode={UseOutlines},plainpages=false,pdfpagelabels=true,
colorlinks=true,linkcolor={black},citecolor={black},urlcolor={black},
pdftitle={\titlestring}, pdfsubject={}, pdfauthor={Masaki Ogura}, pdfkeywords={\mykeywords}}
\usepackage[\hyperrefopt]{hyperref}

\newtheorem{assm}[theorem]{Assumption} 
\newtheorem{remark}[theorem]{Remark}
\newtheorem{problem}[theorem]{Problem}
\newtheorem{example}[theorem]{Example}

\usepackage{soul,color}
\setstcolor{red}
\setul{0.5ex}{0.2ex}
\usepackage{proofread}

\noproofreadmark

\begin{document}

\maketitle

\begin{abstract}
This paper studies the mean stability of positive semi-Markovian jump
linear systems. We show that their mean stability is characterized by
the spectral radius of a matrix that is easy to compute. In deriving the
condition we use a certain discretization of a semi-Markovian jump
linear system that preserves stability. Also we show a characterization
for the exponential mean stability of continuous{-time positive}
Markovian jump linear systems. {Numerical examples are given to
illustrate the results.}
\end{abstract}



\pagestyle{myheadings}
\thispagestyle{plain}
\markboth{M.~Ogura and C.~F.~Martin}{POSITIVE SEMI-MARKOVIAN JUMP LINEAR SYSTEMS}

\section{Introduction}

The stability analysis of switched systems, a class of dynamical systems
whose mathematical structure experiences abrupt changes, is one of the
most fundamental problems in mathematical systems theory
{\cite{Lin2009,Shorten2007,Dayawansa1999}}. In particular, the stability
of positive switched systems, whose state variables are constrained to
be in positive orthants, has received considerable attentions over the
past
decade~\cite{Blan2012,Knorn2009,Fornasini2010a,Shen2012,Ogura2012b}.
{The study of positive switched systems is motivated by their possible
application in pharmacokinetics. In the modern treatment of human
immunodeficiency virus (HIV) infection, multiple drug regimens are
employed to prevent the emergence of drug-resistant
virus~\cite{Wainberg2008}. The authors in \cite{Hernandez-Vargas2011}
solve the minimization problem of such virus mutation by its reduction
to the optimal control problem of a positive switched system under
simplifying assumptions. The reduction was made possible by the
switching nature of HIV treatments and the positivity constraint
naturally placed on the population of virus.} The importance of this
class of switched systems {also} stems from the fact that such
positivity constraints naturally arise in broad areas including
communication systems~\cite{Shorten2006}, formation
flying~\cite{Jadbabaie2003a} and multi agent
systems~\cite{Olfati-Saber2004}.

The stability of positive switched {linear} systems has been mainly
studied by co-positive Lyapunov
functions~\cite{Blan2012,Wu2013,Fornasini2010a,Knorn2009,Gurvits2007}. A
co-positive Lyapunov function is a non-negative linear form of state
variables, which makes a contrast with general cases where the
non-negativity of Lyapunov functions forces us to use quadratic
functionals of state variables. Its linearity often reduces the
stability analysis of positive switched linear systems to linear
problems. For example, a positive switched linear system is stable for
all switching signals if a family of square matrices associated with the
given system consists of matrices whose eigenvalues have only negative
real parts \cite{Knorn2009,Fornasini2010a,Wu2013}.

However, once a switched system is modeled as a \emph{stochastic}
switched system~\cite{Kozin1969,Pola2003}, the above mentioned Lyapunov
function approach fails to take the probability distribution into
account appropriately because it treats any sample path in a rather
equal manner. One of the natural notions of stability in this case is
mean stability \cite{Kozin1969}, which requires that the norm of state
variables converges to $0$ in expectation. One of the earliest results
along this line is by Feng et al.~\cite{Feng1992}, where they give a
necessary and sufficient condition for the exponential mean square
stability of continuous-time Markovian jump linear systems. This result
has been extended to switched {linear} systems with various stochastic
structures including switching Markovian jump linear
systems~\cite{Bolzern2010a}, Markovian jump linear systems with
disturbances \cite{Fragoso2005}, and stochastic hybrid systems with
renewal transitions \cite{Antunes2013}. Stochastic Lyapunov function
approaches for the stability analysis of Markovian jump nonlinear
systems with disturbances can be found
in~\cite{Mao1999,Khasminskii2007}.

The aim of this paper is to give criteria for the mean stability of
positive stochastic switched systems. We assume that the switched
systems are semi-Markovian jump linear systems \cite{Huang2012} (also
called stochastic hybrid systems with renewal transitions
\cite{Antunes2013}), whose switching signal is \del{driven}{\delws}a Markovian renewal
process~\cite{Janssen2006}. We show that their exponential mean
stability is characterized by the spectral radius of a matrix that is
easy to compute. We also allow their state to be reset \cite{Nesic2008}
by random linear mappings at the switching instances.

One of the difficulties in analyzing semi-Markovian jump linear systems
is that its transition rate of discrete-modes is not time-invariant
\cite{Huang2012}. Instead of employing Volterra integral equations used
in \cite{Antunes2013}, we avoid this difficulty by introducing a
discretization of semi-Markovian jump linear systems that admits a
certain time-invariant expression
(Proposition~\ref{proposition:diffeq}). That discretization turns out to
preserve stability properties and hence the system matrix of the
discretization exactly determines the stability of the original system.

This paper is organized as follows. After preparing necessary
mathematical notations, in Section~\ref{section:Hswrt} we give the
definition of continuous-time {positive} semi-Markovian jump linear
systems and state the main result. Section~\ref{section:stabDisc} gives
the stability analysis of discrete-time {positive} semi-Markovian jump
linear systems. Based on the analysis Section~\ref{section:Proof} gives
the proof of the main result. The exponential mean stability of
{positive} Markovian jump linear systems is studied in
Section~\ref{section:mjls}.

\subsection{Mathematical Preliminaries}

Let $(\Omega, \mathcal M, P)$ be a probability space. For a\add{n}
\add{integrable} random variable $X$ on $\Omega$ its expected value is
denoted by $E[X]$. \add{Without being explicitly stated, the random
variables that appear in this paper will be assumed to be integrable.}
If $\mathcal M_1 \subset \mathcal M$ is a $\sigma$-algebra then $E[X\mid
\mathcal M_1]$ denotes the conditional expectation of $X$ given
$\mathcal M_1$. It is well known (see, e.g., \cite{Borkar1995}) that, if
$\mathcal M_2 \subset \mathcal M_1$ is another $\sigma$-algebra then
\begin{equation}\label{eq:filtering}
\begin{aligned}
E[E[X\mid\mathcal M_2] \mid \mathcal M_1]
&=
E[E[X\mid\mathcal M_1] \mid \mathcal M_2]
\\
&=
E[X\mid\mathcal M_2]. 
\end{aligned}
\end{equation}
The $\sigma$-algebra generated by random variables $X_1$, $\dotsc$,
$X_\ell$ is defined \cite{Borkar1995} as the smallest $\sigma$-algebra
on which $X_1$, $\dotsc$, $X_\ell$ are measurable and is denoted by
$\mathcal M(X_1, \dotsc, X_\ell)$.  For a function $f$ on $\mathbb{R}$
its limit at $t$ from the left, if it exists, is denoted by $f(t^-)$.

\begin{lemma}\label{lem:filtering}
Let $\mathcal M_1, \mathcal M_2, \mathcal M_3 \subset \mathcal M$ be
$\sigma$-algebras on $\Omega$ and $X$ be a random variable on~$\Omega$.
If $E[X\mid \mathcal M_1] = E[X\mid \mathcal M_3]$ and $\mathcal M_1
\subset \mathcal M_2 \subset \mathcal M_3$ then $E[X\mid \mathcal M_1] =
E[X\mid \mathcal M_2]$.
\end{lemma}

\begin{proof}
By \eqref{eq:filtering} and the assumption,
\begin{equation*}
\begin{aligned}
E[ \mathchange{f}{X} \mid \mathcal M_1]
&=
E\left[E[X \mid \mathcal M_1]\mid \mathcal M_2\right]
\\
&=
E\left[E[X \mid \mathcal M_3]\mid \mathcal M_2\right]
\\
&=
E\left[X \mid \mathcal M_2\right]. 
\end{aligned}
\end{equation*}
This completes the proof.
\end{proof}

A {real} matrix $A$ is said to be {\it nonnegative} if it has only
nonnegative entries and we write $A\geq 0$. Let $A$ be square. $A$ is
said to be \emph{Metzler} if its off-diagonal entries are nonnegative.
Since any Metzler matrix $A$ can be written as $\tilde A-\alpha I$ where
$\tilde A\geq 0$, $\alpha\in\mathbb{R}$, and $I$ is the identity matrix,
the exponential matrix $e^{At}$ ($t\geq 0$) is nonnegative.  The
Kronecker product \cite{Brewer1978} of $A$ and another matrix $B$ is
denoted by $A\otimes B$.  The spectral radius of $A$ is denoted by
$\rho(A)$. We say that $A$ is Schur stable if $\rho(A) < 1$. {Also the
spectral abscissa of $A$, denoted by $\eta(A)$, is defined as the
largest real part of the eigenvalues of $A$.} {We} say that $A$ is
Hurwitz stable if {$\eta(A) < 0$}.  The next lemma gives basic facts
about nonnegative matrices (see, e.g., \cite{Tam2006}).

\begin{lemma}\label{lemma:PFeig}
Let $A \geq 0$ {be square}. Then $\rho(A)$ is an eigenvalue of $A$ and
there exists a nonnegative eigenvector to the eigenvalue $\rho(A)$.
Moreover if $A\geq B\geq 0$ then $\rho(A)\geq \rho(B)$.
\end{lemma}

The next corollary readily follows from the lemma. 

\begin{corollary}\label{corollary:compSpecRads}
Let $A$ and $B$ be nonnegative square matrices with the same dimensions.
If $A_{ij} > B_{ij}$ whenever $B_{ij} > 0$ then $\rho(A) > \rho(B)$.
\end{corollary}

\begin{proof}
Define $r = \min_{i, j: B_{ij}\neq 0} \frac{A_{ij}}{B_{ij}} > 1$.  Then
$A\geq rB{\geq 0}$ so that, by Lemma~\ref{lemma:PFeig}, $\rho(A) \geq
r\rho(B) > \rho(B)$.
\end{proof}

The $m$-norm of $x\in \mathbb{R}^n$ is defined by $\norm{x}_m =
(\sum_{i=1}^{n} \abs{x_i}^m)^{1/m}$. {The symbol} $1_{\ell}$ denotes the
column vector of length $\ell$ whose entries are all $1$. {The} 1-norm
is linear on the positive orthant $\mathbb{R}^n_+$ because if $x\geq 0$
then  $\onenorm{x} = 1_n^\top x$.  By $e_i$ we denote the $i$-th
standard unit vector in $\mathbb{R}^N$ defined by
\begin{equation*}
[e_i]_j=
\begin{cases}
1 & j=i, \\
0&\text{otherwise. }
\end{cases}
\end{equation*}
It is easy to see that, for all $i${, $m\geq 1$, }and $x\in \mathbb{R}^n$, 
\begin{equation}\label{eq:normdef2}
\norm{x}_m
=
\norm{e_i\otimes x}_m, 
\end{equation}
where the $m$-norm on the right hand side is defined on
$\mathbb{R}^{nN}$.

For $x\in \mathbb{R}^n$ and a positive integer $m$ the vector $x^{[m]}$
is defined \cite{Brockett1973,Barkin1983} as the real vector of length
\begin{equation*}
n_m = \binom{n+m-1}{{m}}
\end{equation*}
whose elements are {all the} lexicographically ordered monomials of
degree $m$ in $x_1$, $\dotsc$, $x_n$. Their nonzero coefficients are
chosen  in such a way that $\twonorm{x}^m = \twonorm{x^{[m]}}$. From
\eqref{eq:normdef2} it follows that
\begin{equation}\label{eq:norm:pi}
\twonorm{e_i \otimes x^{[m]}} = \twonorm{x}^m
\end{equation}
for any standard unit vector $e_i$.   For $A \in\mathbb{R}^{n\times n}$
we define the {$n_m\times n_m$} matrices $A^{[m]}$ and $A_{[m]}$ as the
unique matrices \cite{Brockett1973,Barkin1983} satisfying
\begin{equation}\label{eq:def:A^[m]}
(Ax)^{[m]} = A^{[m]} x^{[m]}
\end{equation}
{for every $x\in\mathbb{R}^n$} and
\begin{equation}\label{eq:def:A_[m]}
\left[\frac{dx}{dt} = Ax\right] 
\Rightarrow 
\left[\frac{dx^{[m]}}{dt} = A_{[m]} x^{[m]}\right]
\end{equation}
for every $\mathbb{R}^n$-valued differentiable function~$x$ \add{on
$\mathbb{R}$}. We need the next two lemmas about the vector $x^{[m]}$.

\begin{lemma}
Let $x$ be an $\mathbb{R}^n$-valued random variable. If $x\geq 0$ with
probability 1 then
\begin{equation}\label{eq:mnorm}
E[\norm{x}_m^m]\leq n\onenorm{E[x^{[m]}]}.
\end{equation}
\end{lemma}

\begin{proof}
By the assumption we have $E[\norm{x}_m^m] = E[\sum^{n}_{i=1}
\abs{x_i}^m] = \sum^{n}_{i=1} E[x_i^m]$. Since $x^{[m]}$ has the entry
$x_i^m$ it holds that $E[x_i^m] = \Abs{E[x_i^m]}\leq
\onenorm{E[x^{[m]}]}$. Thus \eqref{eq:mnorm} is true.
\end{proof}

\begin{lemma}\label{lemma:cone}
The set~$\{e_i \otimes x^{[m]} :  1\leq i\leq N,\ x\in\mathbb{R}^n\}
\subset  \mathbb{R}^{n_mN}$ spans $\mathbb{R}^{n_mN}$ over~$\mathbb{R}$.
\end{lemma}

\begin{proof}
Since each of $e_i$ is a standard unit vector, it is sufficient to show
that the set $S = \{ x^{[m]} : x\in\mathbb{R}^n\}$  spans
$\mathbb{R}^{n_m}$. Assume that $v\in \mathbb{R}^{n_m}$ is orthogonal to
$\spn S$. Then, for every $x\in \mathbb{R}^n$ we have 
\begin{equation}\label{eq:inprodid}
v^\top x^{[m]} = 0. 
\end{equation}
Since $x^{[m]}$ has all the monomials of degree $m$ with nonzero
coefficient,  comparing the coefficient of each monomial in
\eqref{eq:inprodid} shows {$v=0$}, which implies that $\spn S$ is the
entire space $\mathbb{R}^{n_m}$.
\end{proof}

\section{Continuous-time {Positive} Semi-Markovian Jump Linear Systems}
\label{section:Hswrt}

This section introduces continuous-time {positive} semi-Markovian jump
linear systems. Then we define their exponential and stochastic mean
stability. After that we state our main result, which gives the
characterization of the mean stability of continuous-time  positive
semi-Markovian jump linear systems. {A numerical example is given to
illustrate the result.}

Let $A_1$, $\dotsc$, $A_N$ be $n\times n$ real matrices.  Throughout
this paper we fix a probability space $(\Omega, \mathcal M, P)$. Let
$\{\sigma_k\}_{k=0}^\infty$, $\{t_k\}_{k=0}^\infty$, and
$\{J_k\}_{k=0}^\infty$ be stochastic processes on $\Omega$ taking values
in $\{1, \dotsc, N\}$,  $\mathbb{R}_+$, and $\mathbb{R}^{n\times n}$,
respectively. We assume that $\{t_k\}_{k=0}^\infty$ is non-decreasing.
We let
\begin{equation*}
h_k = t_{k+1} - t_k,\ k\geq 0. 
\end{equation*}
Assume that $t_0 = 0$ and $\sigma_0$ is a constant. Let $\Sigma$ be the
stochastic switched system defined by
\begin{equation} \label{eq:SMJLS}
\Sigma: 
\begin{cases}
\dfrac{dx}{dt} = A_{\sigma_k} x(t),\ t_k\leq t<t_{k+1}
\\
x(t_{k+1}) = J_k x(t_{k+1}^-),\ k\geq 0
\end{cases}
\end{equation}
where $x(0) = x_0 \in \mathbb{R}^n$ is a constant vector. 

\begin{definition}
We say that $\Sigma$ is a \emph{continuous-time semi-Markovian jump
linear system} if the following two conditions hold for every $i, j\in
\{1, \dotsc, N\}$, $t\geq 0$, and {every} Borel subset $B$ of
$\mathbb{R}^{n\times n}$.
\begin{enumerate}[label=C{\arabic*}.,ref=C{\arabic*}]
\item \label{item:c:assm:renew} (Markovian property) It holds that 
\begin{equation*}
\begin{aligned}
&P(\sigma_{k+1}=j, h_k \leq t, J_k\in B
\mid \sigma_k, \dotsc, \sigma_{0},t_k, \dotsc, t_0, J_{k-1}, \dotsc, J_0)
\\
=&P(\sigma_{k+1}=j, h_k \leq t, J_k\in B \mid \sigma_k). 
\end{aligned}
\end{equation*}

\item \label{item:c:assm:homo} (Time homogeneity) The probability 
\begin{equation}\label{eq:trapro}
P(\sigma_{k+1}=j, h_k \leq t, J_k\in B \mid \sigma_k = i)
\end{equation}
is independent of $k$. 
\end{enumerate}
Furthermore we say that $\Sigma$ is \emph{positive} if
\begin{enumerate}[label=C{\arabic*}.,ref=C{\arabic*}]\setcounter{enumi}{2}
\item \label{item:c:assm:>=0} (Positivity) The matrices $A_1$, $\dotsc$,
$A_N$ are Metzler and, for each $k\geq 0$, $J_k$ is nonnegative with
probability 1.
\end{enumerate}
\end{definition}

The conditions~\ref{item:c:assm:renew} and~\ref{item:c:assm:homo} {in
particular} show that {the process}~$\{(\sigma_k, t_k)\}_{k=0}^\infty$
is a time-homogeneous Markovian renewal process and therefore
$\{\sigma_k\}_{k=0}^\infty$ is a time-homogeneous Markov
chain~\cite{Janssen2006}. We let  $[p_{ij}]_{ij} \in \mathbb{R}^{N\times
N}$ be the transition matrix of the Markov
chain~$\{\sigma_k\}_{k=0}^\infty$. The condition~\ref{item:c:assm:>=0}
implies that $x(t)\geq 0$ for every $t\geq 0$ with probability 1
provided $x_0 \geq 0$.  Without being explicitly stated, throughout
this paper $\Sigma$ denotes a positive continuous-time semi-Markovian
jump linear system. The aim of this paper is to study the stability of
$\Sigma$ defined as follows.

\begin{definition}\label{definition:c:stbl}
Let $m$ be a positive integer. 
\begin{itemize}
\item $\Sigma$ is said to be \emph{exponentially $m$-th mean stable} if
there exist $C>0$ and $\beta > 0$ such that, for every $x_0$ and
$\sigma_0$,
\begin{equation}\label{eq:def:expsta}
E[\norm{x(t)}^m] \leq Ce^{-\beta t}\norm{x_0}^m. 
\end{equation}

\item $\Sigma$ is said to be \emph{stochastically $m$-th mean stable}
if, for any $x_0$ and $\sigma_0$,
\begin{equation}\label{eq:def:stosta}
\int_0^\infty E[\norm{x(t)}^m]\,dt < \infty. 
\end{equation}
\end{itemize}
\end{definition}

\begin{remark}\label{remark:on:norm}
The stability notions in Definition~\ref{definition:c:stbl} are
independent of the norms used in \eqref{eq:def:expsta} and
\eqref{eq:def:stosta} by the  equivalence of the norms on a
finite-dimensional normed vector space.  Actually we can even use two
different norms in \eqref{eq:def:expsta}.
\end{remark}

We now state the next assumption.  
\begin{assm}\label{assumption:cont}\ 
\begin{enumerate} 
\item \label{item:assm:h_k>0} 
For every $k\geq 0$, 
\begin{equation} \label{eq:hk>0}
h_k > 0. 
\end{equation}

\item \label{item:assm:h_k<T} 
There exists $T > 0$ such that, for every $k\geq 0$, 
\begin{equation}\label{eq:hk<T}
h_k \leq T. 
\end{equation}

\item \label{item:assm:normJ<R} 
There exists $R>0$ such that, for every $k\geq 0$, 
\begin{equation*}
\norm{J_k} \leq R. 
\end{equation*}
\end{enumerate}
\end{assm}

In this assumption, only the second condition, which is also used in
\cite{Antunes2013}, is essential.  The first condition
\ref{item:assm:h_k>0} in Assumption~\ref{assumption:cont} is not
restrictive because {most of} semi-Markovian jump linear system{s} can
be rewritten {as a semi-Markovian jump linear system} that satisfies the
condition~\ref{item:assm:h_k>0}. Also, though in \cite{Antunes2013} they
use constant jump matrices, this paper allows them to be uniformly
bounded random variables.

The next theorem is the main result of this paper. 
\begin{theorem}\label{theorem:main}
$\Sigma$ is exponentially $m$-th mean stable if and only if the block
matrix~$\mathcal A_m (\Sigma) \in \mathbb{R}^{(n_mN)\times (n_mN)}$
whose $(i,j)$-block is defined by
\begin{equation} \label{eq:def:calA}
[\mathcal A_m (\Sigma)]_{ij} 
= 
p_{ji}
E\left[(J_k e^{A_{\sigma_k} h_k})^{[m]} \mid \sigma_k = j, \sigma_{k+1}=i\right]
\in
\mathbb{R}^{n_m\times n_m}
\end{equation}
is Schur stable. 
\end{theorem}

Let us see an example. 
\begin{example}\upshape 
Let $\Sigma$ be a positive semi-Markovian jump linear system with the
two subsystems $\Sigma_1: dx/dt = A_1 x$ and $\Sigma_2: dx/dt = A_2 x$
given by
\begin{equation*}
\begin{gathered}
A_1 = \begin{bmatrix}
-2&0.2\\0.1&-2.3
\end{bmatrix},\
A_2 = \begin{bmatrix}
2.1&0.9\\0.2&0.3
\end{bmatrix}. 
\end{gathered}
\end{equation*}
We assume that $J_k = I$ with probability 1, $p_{11} = p_{22} = 0$, and
$p_{12} = p_{21} = 1$. Suppose that the transition
probability~\eqref{eq:trapro} is given by
\begin{equation}\label{eq:unif}
P(\sigma_{k+1}=1, h_k\leq t \mid \sigma_k = 2) =
\begin{cases}
0,\ t\leq a\\
{(t-a)}/{2a},\ a\leq t\leq 3a\\
1,\ t\geq 3a
\end{cases} 
\end{equation}
where $a>0$ is a constant and
\begin{equation}\label{eq:weibull}
P(\sigma_{k+1}=2, h_k\leq t \mid \sigma_k = 1) =
\begin{cases}
F(t;k,\lambda),\ t\leq t_p\\
1,\ t\geq t_p
\end{cases} 
\end{equation}
where $F(\cdot;k,\lambda)$ denotes the probability distribution function
of the Weibull distribution with the shape parameter $k > 0$ and the
scale parameter $\lambda>0$, i.e., $F(t;k,\lambda) =
1-e^{-(t/\lambda)^k}$ for $t\geq 0$ and $F(t;k,\lambda) = 0$ for $t <
0$, and $t_p > 0$ is the unique number satisfying $F(t_p; k,\lambda) =
1-p$ ($p>0$). $\Sigma$ is clearly positive and satisfies
Assumption~\ref{assumption:cont} with $T = \max(3a, t_p)$ and $R=1$.

{We can regard $\Sigma$ as a controlled system with controller failures.
The stable subsystem~$\Sigma_1$ models the controlled dynamics while the
 unstable one~$\Sigma_2$ models the open dynamics without a controller
in effect. The occurrence of a failure, whose probability is modeled by
the Weibull-like distribution~\eqref{eq:weibull}, gives rise to the
switching from $\Sigma_1$ to $\Sigma_2$. The time it takes for the
controller to be repaired is modeled by the uniform
distribution~\eqref{eq:unif}.}

{We check the first mean and mean square stability of $\Sigma$ using
Theorem~\ref{theorem:main} for the parameters $\lambda =3$, $k=10$, and
$p = 0.1$. Fig.~\ref{fig:Analysis} shows the graph of $\rho(\mathcal
A_1(\Sigma))$ and~$\sqrt{\rho(\mathcal A_2(\Sigma))}$ as the
constant~$a$ moves over $[0.8, 1.2]$.
\begin{figure}[tb]{%
\centering
\includegraphics[width=9cm]{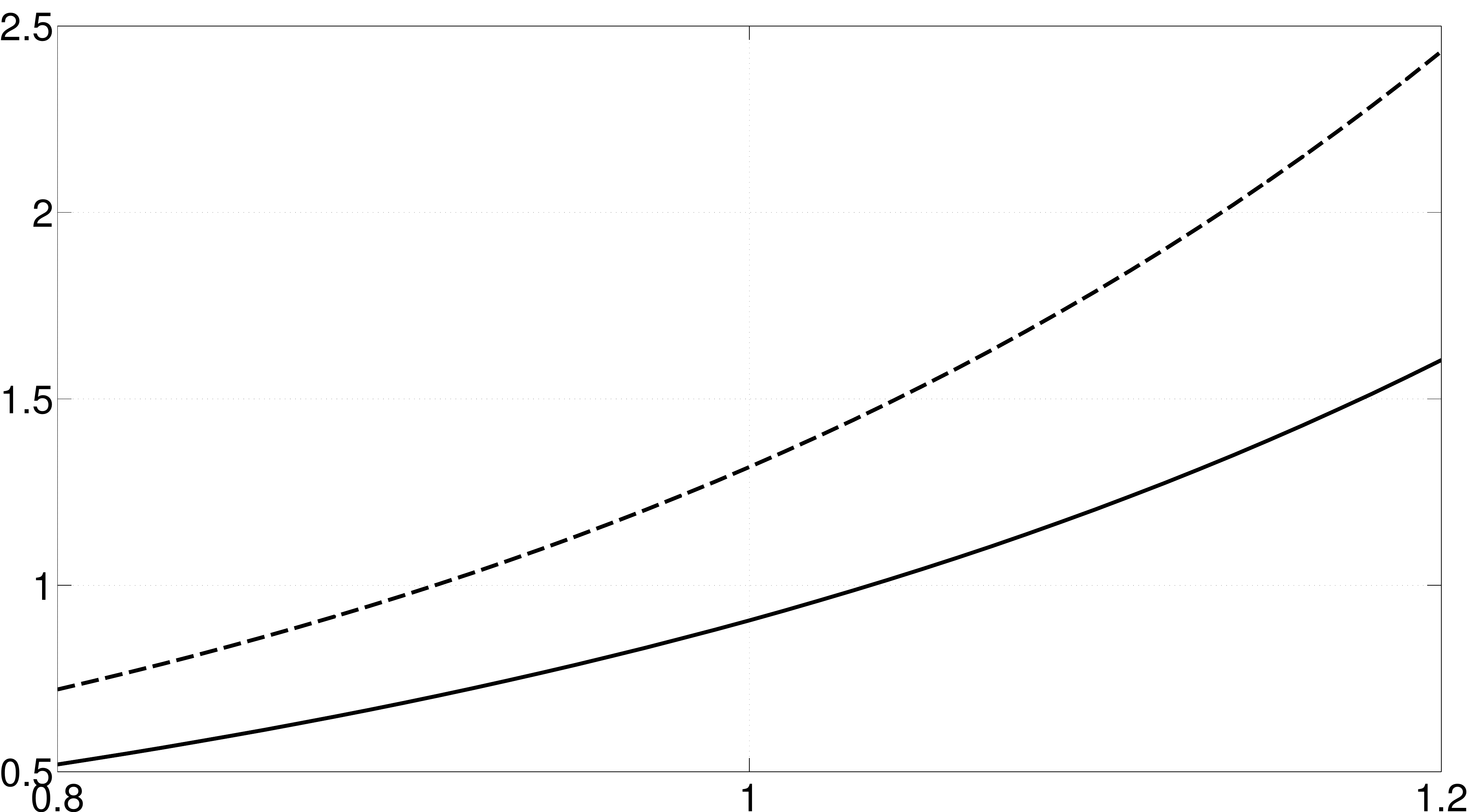}
\caption{Solid: $\rho(\mathcal A_1)$. Dashed: $\sqrt{\rho(\mathcal A_2)}$.}
\label{fig:Analysis}
}\end{figure}
We can see that $\Sigma$ is \add{exponentially} first mean stable if
$a<1.035$ while $\Sigma$ is \add{exponentially} mean square stable only
when $a<0.908$. Notice that the result in \cite{Antunes2013} does not
check the first mean stabilty because it deals with only even
exponents~$m$. A sample path of $\Sigma$ for $a=1$ is shown in
Fig.~\ref{figure:semistabilization}.}

\begin{figure}[tb]{
\centering
\includegraphics[width=9cm]{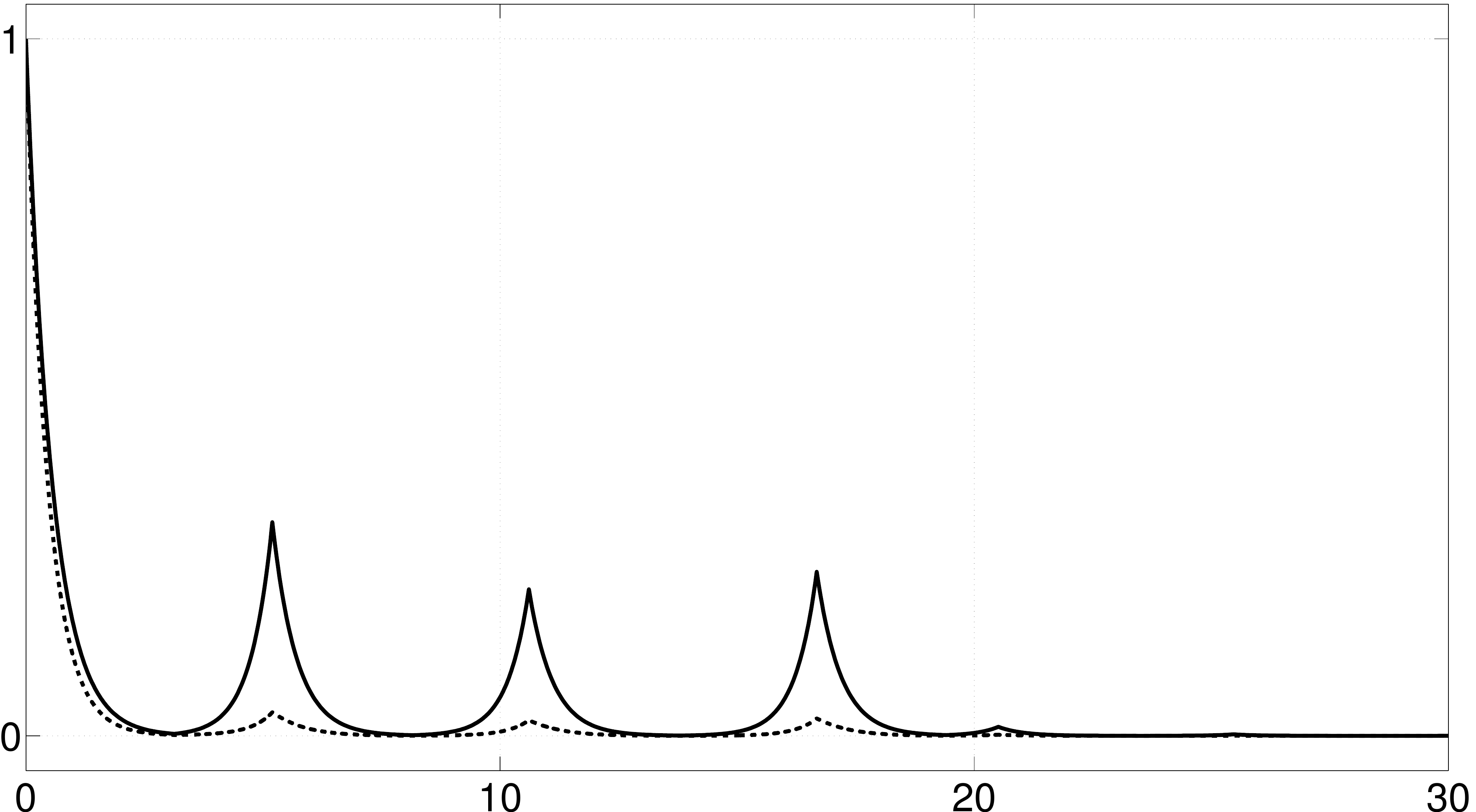}
\caption{A sample path of the semi-Markovian jump linear positive system. Solid: $x_1(t)$. Dotted: $x_2(t)$.} 
\label{figure:semistabilization}
}\end{figure}
\end{example}

We will prove Theorem~\ref{theorem:main} by first investigating the
stability of its discretization. For the trajectory $x$ of $\Sigma$ we
define the discrete-time stochastic process $\{x_d(k)\}_{k=0}^\infty$ by
\begin{equation*}
x_d(k) := x(t_k),\ k\geq 0.
\end{equation*}
Then, by \eqref{eq:hk>0}, $\{x_d(k)\}_{k=0}^\infty$ satisfies 
\begin{equation}\label{eq:def:Ssigma}
\mathcal S \Sigma: 
x_d(k+1) = J_k e^{A_{\sigma_k}h_k} x_d(k),\ k\geq 0.
\end{equation}
In the next section we analyze the stability of a class of stochastic
discrete-time switched systems that include the above defined $\mathcal
S \Sigma$. Based on the analysis Section~\ref{section:Proof} gives  the
proof of the main result of Theorem~\ref{theorem:main}.

\section{Discrete-time {Positive} Semi-Markovian Jump Linear Systems}
\label{section:stabDisc}

Let $\{\sigma_k\}_{k=0}^\infty$ be a time-homogeneous Markov chain
taking values in $\{1, \dotsc, N\}$ with the probability transition
matrix~$[p_{ij}]_{ij}$. We assume that $\sigma_0$ is a constant. Let
$\{F_k\}_{k=0}^\infty$ be another stochastic process on $\Omega$ taking
values in $\mathbb{R}^{n\times n}$. Define the discrete-time switched
system $\Sigma_d$ by
\begin{equation*}
\Sigma_d: x_d(k+1) = F_kx_d(k),\ x_d(0) = x_0. 
\end{equation*}

\begin{definition} \label{definition:d}
We say that $\Sigma_d$ is a \emph{discrete-time semi-Markovian jump
linear system} if the following \change{three}{two} conditions hold for every $k\geq
0$, $i, j \in \{1, \dotsc, N\}$, and every Borel subset~$B$ of
$\mathbb{R}^{n\times n}$.
\begin{enumerate}[label=D{\arabic*}.,ref=D{\arabic*}] 
\item (Markovian property) \label{item:d:assm:renew} It holds that
\begin{equation*}
P(\sigma_{k+1} = j, F_k \in B \mid \sigma_k, \dotsc, \sigma_0, F_{k-1}, \dotsc, F_0) =
P(\sigma_{k+1} = j, F_k \in B \mid \sigma_k). 
\end{equation*}

\item \label{item:d:assm:homo} (Time homogeneity) The expected
probability
\begin{equation*}
P(\sigma_{k+1} = j,  F_k \in B \mid \sigma_k = i)
\end{equation*}
does not depend on $k$. 
\end{enumerate}
Furthermore we say that $\Sigma_d$ is \emph{positive} if
\begin{enumerate}[label=D{\arabic*}.,ref=D{\arabic*}]\setcounter{enumi}{2}
\item \label{item:d:assm:>=0} (Positivity) $F_k$ is nonnegative with
probability 1. 
\end{enumerate}
\end{definition}

Each of the conditions~\ref{item:d:assm:renew}, \ref{item:d:assm:homo},
and~\ref{item:d:assm:>=0} corresponds to \ref{item:c:assm:renew},
\ref{item:c:assm:homo}, and~\ref{item:c:assm:>=0} in
Definition~\ref{definition:c:stbl}, respectively. Throughout this
section $\Sigma_d$ denotes a positive discrete-time semi-Markovian jump
linear system.

The stability of discrete-time semi-Markovian jump linear systems is
defined in a similar way as that of continuous-time ones.  As mentioned
in Remark~\ref{remark:on:norm}, any combination of norms can be also
used in the following definition.
\begin{definition} Let $m$ be a positive integer.
\begin{itemize}
\item $\Sigma_d$ is said to be \emph{exponentially $m$-th mean stable}
if there exist $C>0$ and $\beta > 0$ such that, for any $x_0$ and
$\sigma_0$, 
\begin{equation}\label{eq:def:d:expsta}
E[\norm{x_d(k)}^m] \leq Ce^{-\beta k} \norm{x_0}^m. 
\end{equation}

\item $\Sigma_d$ is said to be \emph{stochastically $m$-th mean stable}
if, for any $x_0$ and $\sigma_0$, 
\begin{equation*}
\sum_{k=0}^\infty E[\norm{x_d(k)}^m] < \infty. 
\end{equation*}
\end{itemize}
\end{definition}

We place the following assumption that corresponds to the
conditions~\ref{item:assm:h_k<T} and~\ref{item:assm:normJ<R} of
Assumption~\ref{assumption:cont}.
\begin{assm}\label{assumption:disc}
The expected value 
\begin{equation}\label{eq:calF.ij}
E[F_k^{{[m]}} \mid \sigma_k = i, \sigma_{k+1} = j]
\end{equation} exists for all $i, j\in \{1, \dotsc, N\}$. 
\end{assm}

Notice that, by \ref{item:d:assm:homo}, the expected
value~\eqref{eq:calF.ij} does not depend on $k$. The next theorem gives
a characterization of the stability of $\Sigma_d$ and is used in
Section~\ref{section:Proof} to prove Theorem~\ref{theorem:main}.

\begin{theorem} \label{theorem:d:stbl} 
The following statements are equivalent: 
\begin{enumerate}
\item \label{item:d:expstbl}
$\Sigma_d$ is exponentially $m$-th mean stable.

\item \label{item:d:stostable}
$\Sigma_d$ is stochastically $m$-th mean stable.

\item \label{item:d:Schur}  
The block matrix $\mathcal F_m \in \mathbb{R}^{(n_mN)\times (n_mN)}$ 
whose $(i,j)$-block is defined by
\begin{equation}\label{eq:def:calF}
[\mathcal F_m]_{ij} 
= 
p_{ji} E[F_k^{[m]} \mid \sigma_k = j, \sigma_{k+1} = i]
\in
\mathbb{R}^{n_m\times n_m}
\end{equation}
is Schur stable. 
\end{enumerate}
\end{theorem}

\begin{remark}
Theorem~\ref{theorem:d:stbl} extends  the stability characterizations of
discrete-time Markovian jump linear systems given in
\cite{Fang2002a,Costa2004} to semi-Markovian jump linear systems.
\end{remark}

The rest of this section is devoted to the proof of Theorem
\ref{theorem:d:stbl}. Let us first observe that, by the positivity of
$\Sigma_d$, the initial state $x_0$ can be assumed to be nonnegative
without loss of generality.

\begin{lemma}\label{lemma:x0>0wlog}
$\Sigma_d$ is exponentially $m$-th mean stable if and only if there
exist $C>0$ and $\beta>0$ such that \eqref{eq:def:d:expsta} holds for
any $x_0\geq 0$ and $\sigma_0$.
\end{lemma}

\begin{proof}
The necessity part is obvious. Let us prove the sufficiency part. Assume
that there exist $C>0$ and $\beta>0$ such that \eqref{eq:def:d:expsta}
holds for any $x_0\geq 0$ and $\sigma_0$.  Let $x_0\in\mathbb{R}^n$ be
arbitrary. Define $x_0^+, x_0^- \in \mathbb{R}^n_+$ as the entry-wise
maximum and minimum $x_0^+ := \max(x_0, 0)$ and $x_0 := \max(-x_0, 0)$.
Let $x^+_d(k)$ and $x^-_d(k)$ denote the solutions of $\Sigma_d$ with
the initial states $x_0^+$ and $x_0^-$, respectively. Since $x_0 = x_0^+
- x_0^-$ we have $x_d(k) = x_d^+(k) - x_d^-(k)$.  Then the general
inequalities $(a+b)^m \leq 2^m(a^m + b^m)$ ($a, b\geq 0$) and
$\norm{x_0^+}_2^m +  \norm{x_0^-}_2^m \leq 2\norm{x_0}^m_2$ show that
\begin{equation*}
\begin{aligned}
E[\norm{x_d(k)}^m_2]
&\leq
E\left[(\norm{x_d(k;x_0^+)}_2 + \norm{x_d(k;x_0^-)}_2)^m\right]
\\
&\leq
2^m E\left[\norm{x_d(k;x_0^+)}_2^m + \norm{x_d(k;x_0^-)}_2^m\right]
\\
&\leq
2^m Ce^{-\beta k} (\norm{x_0^+}_2^m + \norm{x_0^-}_2^m)
\\
&\leq
2^{m+1} Ce^{-\beta k} \norm{x_0}_2^m. 
\end{aligned}
\end{equation*}
Therefore $\Sigma_d$ is exponentially $m$-th mean stable. 
\end{proof}

Let us introduce the stochastic process $\{\zeta(k)\}_{k=0}^\infty$
taking values in the set of standard unit vectors of $\mathbb{R}^N$ and
defined by
\begin{equation*}
\zeta(k) = e_{\sigma_k}. 
\end{equation*}
We will need the next technical lemma.

\begin{lemma}\label{lemma:indp}
Let $k\geq 0$ and $i, j\in \{1, \dotsc, N\}$ be arbitrary. Assume
$P(\sigma_k =i)\neq 0$.  Define the  probability space $(\Omega',
\mathcal M', P')$ by
\begin{equation}\label{eq:Omega'}
\begin{aligned}
\Omega' & = \{\omega\in \Omega: \sigma_k = i\}, \\
\mathcal M' &= \{M' \subset \Omega':  M' \in \mathcal M\}, \\
P'(M') &= P(M')/P(\Omega'). 
\end{aligned}
\end{equation}
Then the random variables $\zeta(k+1)_j F_k^{[m]}$ and $x_d(k)^{[m]}$
are independent on $\Omega'$.
\end{lemma}

\begin{proof}
See Appendix~\ref{appendix:proof:lemma:indp}. 
\end{proof}

This lemma proves the next proposition, which plays the key role in the
proof of Theorem~\ref{theorem:d:stbl}.

\begin{proposition}\label{proposition:diffeq}
{The matrix $\mathcal F_m$ is nonnegative.} {Moreover,
for} every $k\geq 0$,
\begin{equation}\label{eq:diff:calF}
E\left[\zeta(k+1)\otimes x_d(k+1)^{[m]}\right] 
= 
\mathcal F_m E\left[\zeta(k)\otimes x_d(k)^{[m]}\right]. 
\end{equation}
\end{proposition}
\begin{proof} 
The nonnegativity of $\mathcal F_m$ is clear from~\ref{item:d:assm:>=0}.
To show \eqref{eq:diff:calF} let us fix arbitrary $x_0$ and $\sigma_0$.
By the definition of Kronecker products,
\begin{equation}\label{eq:Kron:decom}
\zeta(k+1)\otimes x_d(k+1)^{[m]}
=
\begin{bmatrix}
\zeta(k+1)_1 x_d(k+1)^{[m]}
\\
\vdots
\\
\zeta(k+1)_N x_d(k+1)^{[m]}
\end{bmatrix}. 
\end{equation}
Now we recall $\sum_{i=1}^N \zeta(k)_i = 1$. Since the
equation~\eqref{eq:def:A^[m]} gives $x_d(k+1)^{[m]} = F_k^{[m]}
x_d(k)^{[m]}$ we have   $x_d(k+1)^{[m]} = \sum_{i=1}^N \zeta(k)_i
F_{k}^{[m]}x_d(k)^{[m]}$.  Therefore, for a fixed $j$,
\begin{equation}\label{eq:E[zeta.x_d]}
E\left[\zeta(k+1)_j x_d(k+1)^{[m]}\right] 
= 
\sum_{i=1}^N 
E\left[\zeta(k+1)_j \zeta(k)_i F_k^{[m]} x_d(k)^{[m]}\right]. 
\end{equation}
Let us temporally fix also $i$ and assume that $P(\sigma_k = i) \neq 0$.
Let $(\Omega', \mathcal M', P')$ be the probability space defined by
\eqref{eq:Omega'} and let $E'[\cdot]$ denote the expected value on
$\Omega'$. Notice that, for a random variable $X$ on $\Omega$, we have
\begin{equation*}
E[\zeta(k)_i X] 
= 
E[1_{\{\sigma_k = i\}} X] 
=
E'[X]P(\sigma_k = i). 
\end{equation*}
Then, Lemma~\ref{lemma:indp} yields that 
\begin{equation}\label{eq:E[long]}
\begin{aligned}
E\left[\zeta(k+1)_j \zeta(k)_i F_k^{[m]} x_d(k)^{[m]}\right]
&=
P(\sigma_k = i)E'\left[\zeta(k+1)_j F_k^{[m]} x_d(k)^{[m]}\right]
\\
&=
P(\sigma_k = i)E'\left[\zeta(k+1)_j F_k^{[m]}\right] E'\left[x_d(k)^{[m]}\right]
\\
&=
\frac{E\left[\zeta(k+1)_j F_k^{[m]} \zeta(k)_i\right]}{P(\sigma_k = i)} 
E\left[\zeta(k)_i x_d(k)^{[m]}\right]. 
\end{aligned}
\end{equation}
Since 
\begin{equation*}
\begin{aligned}
E\left[\zeta(k+1)_j F_k^{[m]} \zeta(k)_i\right]
&=
E\left[F_k^{[m]} 1_{\{\sigma_k = i, \sigma_{k+1} = j\}}\right]
\\
&=
P(\sigma_k = i, \sigma_{k+1} = j)
E\left[F_k^{[m]} \mid \sigma_k = i, \sigma_{k+1} = j\right], 
\end{aligned}
\end{equation*}
by \eqref{eq:E[long]} we have 
\begin{equation*}
\begin{aligned}
E\left[\zeta(k+1)_j \zeta(k)_i F_k^{[m]} x_d(k)^{[m]}\right]
&=
p_{ij} E[F_k^{[m]} \mid \sigma_k = i, \sigma_{k+1} = j] 
E\left[\zeta(k)_i x_d(k)^{[m]}\right]
\\
&=
[\mathcal F_m]_{ji} E\left[\zeta(k)_i x_d(k)^{[m]}\right], 
\end{aligned}
\end{equation*}
where $[\mathcal F_m]_{ji}$ is defined by \eqref{eq:def:calF}.  Notice
that this equation is valid even when $P(\sigma_k = i) = 0$ because, in
this case, the left and right hand sides of the equation are both $0$.
Hence, by \eqref{eq:E[zeta.x_d]},
\begin{equation*}
E[\zeta(k+1)_j x_d(k+1)^{[m]}] = 
\begin{bmatrix}
 [\mathcal F_m]_{j1} & \cdots& [\mathcal F_m]_{jN}
\end{bmatrix}
E[\zeta(k)\otimes x_d(k)^{[m]}]. 
\end{equation*}
This equation and \eqref{eq:Kron:decom} prove \eqref{eq:diff:calF}. 
\end{proof}

Now we prove Theorem~\ref{theorem:d:stbl}. 

\defproof{Theorem~\ref{theorem:d:stbl}}
\begin{proof}
Let us show the cycle [\ref{item:d:expstbl} $\Rightarrow$
\ref{item:d:stostable} $\Rightarrow$~\ref{item:d:Schur} $\Rightarrow$
\ref{item:d:expstbl}].

[\ref{item:d:expstbl} $\Rightarrow$~\ref{item:d:stostable}]:\quad If
$\Sigma_d$ is exponentially $m$-th mean stable then $\Sigma_d$ is
clearly stochastically $m$-th mean stable. 

[\ref{item:d:stostable} $\Rightarrow$~\ref{item:d:Schur}]:\quad Suppose
that $\Sigma_d$ is stochastically $m$-th mean stable. Assume
$\rho(\mathcal F_m)\geq 1$ to derive a contradiction. Since $\mathcal
F_m$ is nonnegative it has an eigenvector $v$ corresponding to the
eigenvalue $\rho(\mathcal F_m)$ by Lemma~\ref{lemma:PFeig}.  For this
vector $v$, by Lemma~\ref{lemma:cone} there exist  $c_1$, $\dotsc$,
$c_\ell \in \mathbb{R}$,  $\sigma_0^{(1)}, \dotsc, \sigma_0^{(\ell)} \in
\{1, \dotsc, N\}$, and $x_0^{(1)}, \dotsc, x_0^{(\ell)} \in
\mathbb{R}^n$ such that $v = \sum_{i=1}^\ell c_i \zeta_0^{(i)} \otimes
(x_0^{(i)})^{[m]}$. Let $x_d^{(i)}(k)$ be the solution of $\Sigma_d$
with the initial state $x_0^{(i)}$ and mode $\sigma_0^{(i)}$.  Then, by
\eqref{eq:diff:calF},
\begin{equation*}
\sum_{i=1}^\ell 
c_i E[\zeta^{(i)}(k) \otimes x_d^{(i)}(k)^{[m]}]
=
\mathcal F_m^k v
=
\rho(\mathcal F_m)^k v.
\end{equation*}
Therefore, by \eqref{eq:norm:pi}, 
\begin{equation*}
\begin{aligned}
\rho(\mathcal F_m)^k \norm{v}_2
&\leq
\sum_{i=1}^\ell \abs{c_i}
E\left[
\twonorm{\zeta^{(i)}(k) \otimes x_d^{(i)}(k)^{[m]}}
\right]
\\
&=
\sum_{i=1}^\ell 
\abs{c_i} E\left[\twonorm{x_d^{(i)}(k)}^m\right]
\end{aligned}
\end{equation*}
because $\zeta^{(i)}(k)$ takes its value in the set of standard unit
vectors.  Thus, the stochastic $m$-th mean stability of $\Sigma_d$ shows
that
\begin{equation*}
\begin{aligned}
\sum_{k=0}^{\infty} \rho(\mathcal F_m)^k \norm{v}_2
&\leq
\sum_{i=1}^\ell \abs{c_i}  \sum_{k=0}^{\infty} 
E\left[\twonorm{x_d^{(i)}(k)}^m\right] 
< 
\infty. 
\end{aligned}
\end{equation*}
This gives a contradiction because $\rho(\mathcal F_m)\geq 1$ and $v\neq
0$.

{[\ref{item:d:Schur} $\Rightarrow$
\ref{item:d:expstbl}]}:\quad    Suppose $\rho(\mathcal F_m) < 1$. Then
we can take \cite[Lemma~5.6.10]{Horn1990} a norm $\nnnorm{\cdot}$ on
$\mathbb{R}^{n_mN}$ such that 
\begin{equation*} 
\nnnorm{\mathcal F_m} < 1. 
\end{equation*}
Let $x_0 \geq 0$ and $\sigma_0$ be arbitrary. Since $x_d(k)$  is
nonnegative, \eqref{eq:mnorm} shows that 
\begin{equation}\label{eq:tech1}
E\left[ \norm{x_d(k)}_m^m \right] 
\leq 
n \oneNorm{E[x_d(k)^{[m]}]}. 
\end{equation}
Since $\sum_{i=1}^N \zeta(k)_i = 1$ and the 1-norm is linear on a
positive orthant, 
\begin{equation}\label{eq:tech2}
\begin{aligned}
\oneNorm{E[x_d(k)^{[m]}]}
&=
\oneNorm{E\left[\sum_{i=1}^{m} \zeta(k)_ix_d(k)^{[m]}\right]}
\\
&=
\sum_{i=1}^{m} \oneNorm{E[\zeta(k)_ix_d(k)^{[m]}]}
\\
&=
\oneNorm{
\begin{bmatrix}
E\left[\zeta(k)_1 x_d(k)^{[m]}\right]
\\
\vdots
\\
E\left[\zeta(k)_N x_d(k)^{[m]}\right]
\end{bmatrix}
}
\\
&=
\onenorm{E[\zeta(k) \otimes x_d(k)^{[m]}]}
\\
&=
\onenorm{\mathcal F_m^k (\zeta_0 \otimes x_0^{[m]})}. 
\quad
\text{(by \eqref{eq:diff:calF})}
\end{aligned}
\end{equation}
By the equivalence of the norms on a finite-dimensional vector space,
there exist positive constants $C_1$ and $C_2$ such that 
\begin{equation*}
\begin{aligned}
\onenorm{\mathcal F_m^k (\zeta_0 \otimes x_0^{[m]})}
&\leq
C_1\nnnorm{\mathcal F_m^k (\zeta_0\otimes x_0^{[m]})}
\\
&\leq
C_1\nnnorm{\mathcal F_m}^k \nnnorm{\zeta_0\otimes x_0^{[m]}}
\\
&\leq
C_1 C_2 \nnnorm{\mathcal F_m}^k \twonorm{\zeta_0\otimes x_0^{[m]}}
\\
&=
C_1 C_2 \nnnorm{\mathcal F_m}^k \twonorm{x_0}^m.
\quad 
\text{(by \eqref{eq:norm:pi})}
\end{aligned}
\end{equation*}
This inequality together with \eqref{eq:tech1} and \eqref{eq:tech2}
proves the exponential convergence \eqref{eq:def:d:expsta} because
$\nnnorm{\mathcal F_m} < 1$. Thus, since $x_0\geq 0$ and $\sigma_0$ were
arbitrary, Lemma~\ref{lemma:x0>0wlog} shows the exponential $m$-th mean
stability of $\Sigma_d$.
\end{proof}
\undefproof

\section{Proof of the Main Result}
\label{section:Proof}

This section gives the proof of Theorem~\ref{theorem:main}.  We
separately prove sufficiency and necessity. Let $\Sigma$ be a
continuous-time positive semi-Markovian jump linear system satisfying
Assumption~\ref{assumption:cont} and let $\mathcal S \Sigma$ be its
discretization defined by \eqref{eq:def:Ssigma}.

\subsection{Proof of Sufficiency}

Let us begin by checking that the discrete-time system $\mathcal S
\Sigma$ is a discrete-time {positive} semi-Markovian jump linear
system. 

\begin{lemma}\label{lemma:satisfiesAssm}
$\mathcal S \Sigma$ is a discrete-time positive semi-Markovian jump
linear system {satisfying Assumption~\ref{assumption:disc}}.
\end{lemma}

\begin{proof}
Let $\{\sigma_k\}_{k=0}^\infty$, $\{t_k\}_{k=0}^\infty$, and
$\{J_k\}_{k=0}^\infty$ be the stochastic processes defining $\Sigma$ and
define $F_k = J_k e^{A_{\sigma_k} h_k}$.  Let us show the first
condition~\ref{item:d:assm:renew}.  Define the $\sigma$-algebras
$\mathcal M_1$, $\mathcal M_2$, and $\mathcal M_3$ on $\Omega$ by
\begin{equation*}
\begin{aligned}
\mathcal M_1
&= 
\mathcal M(\sigma_k),
\\
\mathcal M_2 
&= 
\mathcal M(\sigma_k, \dotsc, \sigma_0, F_{k-1}, \dotsc, F_0),
\\
\mathcal M_3 
&= 
\mathcal M(\sigma_k, \dotsc, \sigma_0, t_k, \dotsc, t_0, J_{k-1}, \dotsc, J_0). 
\end{aligned}
\end{equation*}
 {Then one can see that
\begin{equation}\label{eq:3tower}
\mathcal M_1
\subset \mathcal M_2 
\subset \mathcal M_3. 
\end{equation}
The first part of this inclusion is obvious. To show the second part it
is sufficient to show that the function~$F_\ell$ is measurable on
$\mathcal M_3$ for every $\ell=0$, $\dotsc$, $k-1$. Let us fix an
$\ell$. Since both $\sigma_\ell$ and $h_\ell$ are measurable on
$\mathcal M_3$ so is $A_{\sigma_\ell}h_\ell$. Since the matrix
exponential function on $\mathbb{R}^{n\times n}$ is continuous, the
mapping~$e^{A_{\sigma_\ell}h_\ell}\colon \Omega \to \mathbb{R}^{n\times
n}$ is measurable on $\mathcal M_3$. Thus the measurability of $J_\ell$
actually proves that $F_\ell$ is measurable on $\mathcal M_3$, which
completes the proof of $\mathcal M_2 \subset \mathcal M_3$. Now let $j
\in \{1, \dotsc, N\}$ and a Borel set $B\subset \mathbb{R}^{n\times n}$
be arbitrary and define $f = \chi_{ \{\sigma_{k+1} = j, F_k \in B\} } $.
Since~\ref{item:c:assm:renew} shows $E[f \mid \mathcal M_3] = E[f \mid
\mathcal M_1]$, from Lemma~\ref{lem:filtering} and \eqref{eq:3tower} we
can see $E[f \mid \mathcal M_2] = E[f\mid \mathcal M_1]$, which
immediately proves \ref{item:d:assm:renew}.}

The second condition~\ref{item:d:assm:homo} is obviously true {by
\ref{item:c:assm:homo}} because $F_k$ is a measurable function of
$\sigma_k$, $h_k$, and $J_k$. \ref{item:d:assm:>=0} follows from the
positivity condition~\ref{item:c:assm:>=0}. Finally
Assumption~\ref{assumption:disc} holds by  the
conditions~\ref{item:assm:h_k<T} and~\ref{item:assm:normJ<R} of
Assumption~\ref{assumption:cont}.
\end{proof}

The next corollary immediately follows from the
definition~\eqref{eq:def:calA} of the matrix $\mathcal A_m(\Sigma)$,
Theorem~\ref{theorem:d:stbl}, and Lemma~\ref{lemma:satisfiesAssm}. 

\begin{corollary}\label{corollary:Schur<->DiscSta}
The following statements are equivalent. 
\begin{itemize}
\item $\mathcal S\Sigma$ is exponentially $m$-th mean stable.

\item $\mathcal S\Sigma$ is stochastically $m$-th mean stable.

\item $\mathcal A_m(\Sigma)$ is Schur stable. 
\end{itemize}
\end{corollary}

The next lemma helps us to relate the stability of $\Sigma$ and
$\mathcal S \Sigma$ and will be used repeatedly.

\begin{lemma}
There exist $0<C_1<C_2<\infty$ such that, for every sample path $x$
of~$\Sigma$ and $k\geq 0$, 
\begin{equation}\label{eq:C1}
C_1 \norm{x(t_k)} \leq \norm{x(t)},\ t_k\leq t< t_{k+1}
\end{equation}
and 
\begin{equation}\label{eq:C2}
\norm{x(t)} \leq C_2\norm{x(t_k)},\ t_k\leq t\leq t_{k+1}.
\end{equation}
\end{lemma}

\begin{proof}
First let $t\in [t_k, t_{k+1})$ be arbitrary. Then there exist $h \in
[0, T]$ and $i\in \{1, \dotsc, N\}$ such that $x(t) = e^{A_i
h}x(t_k)$ and therefore $x(t_k) = e^{-A_i h}x(t)$. Hence
\begin{equation*}
\begin{aligned}
\norm{x(t_k)} 
&\leq e^{\norm{A_i} h} \norm{x(t)}
\\
&\leq e^{\max_{1\leq i\leq N} \norm{A_i} T} \norm{x(t)}
\end{aligned}
\end{equation*} 
so that $C_1 :=e^{-\max_{1\leq i\leq N} \norm{A_i} T}$ satisfies
\eqref{eq:C1}. 

Then let $t\in [t_k, t_{k+1}]$. Then there exist $h \in [0, T]$, $i\in
\{1, \dotsc, N\}$, and $J\in \mathbb{R}^{n\times n}$ such that $x(t) =
Je^{A_i h}x(t_k)$. By Assumption~\ref{assumption:cont},
\begin{equation*}
\begin{aligned}
\norm{x(t)} 
&\leq 
\norm{J} e^{\norm{A_i} h} \norm{x(t_k)}
\\
&\leq 
R e^{\max_{1\leq i\leq N} \norm{A_i} T}\norm{x(t_k)}. 
\end{aligned}
\end{equation*}
Thus the inequality \eqref{eq:C2} holds for $C_2 = R e^{\max_{1\leq
i\leq N} \norm{A_i} T}$. 
\end{proof}

Let us prove the sufficiency part of Theorem~\ref{theorem:main}. 

\defproof{the sufficiency part of Theorem~\ref{theorem:main}}
\begin{proof}
Assume that $\mathcal A_m(\Sigma)$ is Schur stable. Then, by Corollary
\ref{corollary:Schur<->DiscSta}, $\mathcal S\Sigma$ is exponentially
$m$-th mean stable. We shall show that $\Sigma$ is exponentially $m$-th
mean stable. To this end, for $t\geq 0$, define the random
variable~$k_t$ by
\begin{equation*}
k_{t}(\omega) = \max\{ k \in \mathbb{N}: t_k(\omega)\leq t \}. 
\end{equation*}
Notice that \eqref{eq:hk<T} shows $t<t_{k_{t} + 1} \leq T(k_{t}
+1)$ and therefore
\begin{equation*}
k_{t} > T^{-1}t -1. 
\end{equation*}

Let $x_0$, $\sigma_0$, and $t\geq 0$ be arbitrary. Let $x$ be the
trajectory of $\Sigma$ and define $x_d(k) = x(t_k)$. Since $t_{k_t} \leq
t<t_{k_t+1}$, the inequality \eqref{eq:C2} gives $\norm{x(t)} \leq C_2
\norm{x(t_{k_t})}= C_2 \norm{x_d(k_{t})}$. Therefore
\begin{equation*}
\begin{aligned}
E[\norm{x(t)}^m]
&\leq
C_2\int_{\Omega} \norm{x_d(k_t)}^m\,dP
\\
&=
C_2\sum_{\ell > T^{-1}t -1} 
\int_{\{\omega: k_t = \ell \}} 
\norm{x_d(\ell)}^m\,dP
\\
&\leq
C_2\sum_{\ell > T^{-1}t -1} 
E[\norm{x_d(\ell)}^m]
\\
&\leq
C_2\sum_{\ell > T^{-1}t -1} 
Ce^{-\beta \ell} \norm{x_0}^m
\\
&\leq
\frac{CC_2e^\beta}{1-e^{-\beta}} e^{-\beta T^{-1}t} \norm{x_0}^m. 
\end{aligned}
\end{equation*}
Thus $\Sigma$ is exponentially $m$-th mean stable.
\end{proof}
\undefproof

\subsection{Proof of Necessity}

Then let us prove necessity for Theorem~\ref{theorem:main}. For the
proof we use a family of continuous-time semi-Markovian jump linear
systems~$\Sigma^{(\tau)}$ ($\tau > 0$) defined by
\begin{equation*}
\begin{aligned}
\Sigma^{(\tau)}: 
\begin{cases}
\dfrac{dx}{dt} = A_{\sigma_k} x(t),\ t_k\leq t<t_{k+1}
\\
x(t_{k+1}) = J^{(\tau)}_k x(t_{k+1}^-),\ k\geq 0
\end{cases}
\end{aligned}
\end{equation*}
where 
\begin{equation*}
J^{(\tau)}_k :=  \chi_{\{h_k\geq \tau\}} J_k
\end{equation*}
for each $k$. Roughly speaking, $\Sigma^{(\tau)}$ makes its state jump
to $0$ whenever it observes a dwell time $h_k$ less than $\tau$ (see
Fig.~\ref{figure:SigmaTau}).
\begin{figure}[tb]
\centering
\includegraphics[width=8cm]{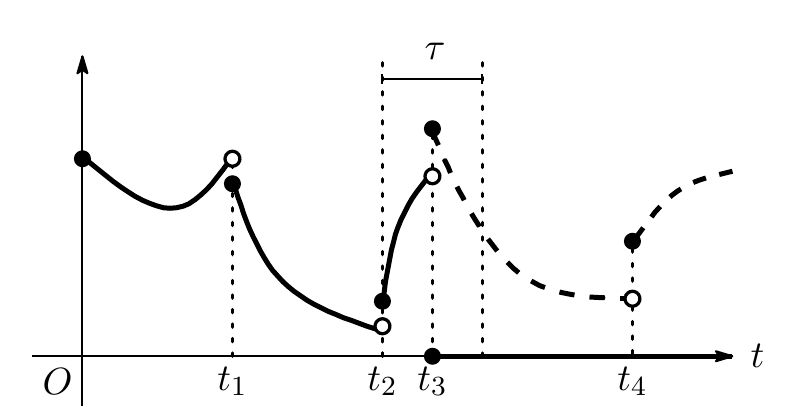}
\caption{Sample paths of $\Sigma$ (dashed) and $\Sigma^{(\tau)}$ (solid)}
\label{figure:SigmaTau}
\end{figure}
 
The next proposition shows that the switched system $\mathcal
S\Sigma^{(\tau)}$ inherits the stability of $\Sigma$. 

\begin{proposition}\label{proposition:tauSta}
{Let $\tau>0$ be arbitrary.} If $\Sigma$ is stochastically $m$-th
mean stable then $\mathcal S \Sigma^{(\tau)}$ is also stochastically
$m$-th mean stable.
\end{proposition}

\begin{proof}
Assume that $\Sigma$ is stochastically $m$-th mean stable. Let us show
that $\mathcal S \Sigma^{(\tau)}$ is stochastically $m$-th mean stable.
Let $x$ be a sample path of $\Sigma$ and define $x_d(k) = x(t_k)$. First
assume that there exists $k\geq 0$ such that $h_k < \tau$. Let $k_0$ be
the minimum of such $k$. Then $x_d(k) = 0$ for every $k>k_0$ by the
definition of $J_k^{(\tau)}$.  Therefore, by \eqref{eq:C2},
\begin{equation*}
\begin{aligned}
\sum_{k=0}^{\infty} \norm{x_d(k)}^m
&=
\norm{x(t_{k_0})}^m + \sum_{k=0}^{k_0-1}\norm{x(t_k)}^m
\\
&\leq
C_2^m\norm{x(t_{k_0-1})}^m + \sum_{k=0}^{k_0-1}\norm{x(t_k)}^m
\\
&\leq
(C_2^m + 1)\sum_{k=0}^{k_0-1}\norm{x(t_k)}^m. 
\end{aligned}
\end{equation*}
Since $h_k \geq \tau$ and therefore $1 \leq
\tau^{-1}\int_{t_k}^{t_{k+1}} dt$ for every $k=0, \dotsc, k_0-1$, 
using \eqref{eq:C1} we can show that 
\begin{equation} \label{eq:tausta1}
\begin{aligned}
\sum_{k=0}^{k_0-1}\norm{x(t_k)}^m
&\leq
\tau^{-1}\sum_{k=0}^{k_0-1} \int_{t_k}^{t_{k+1}}\norm{x(t_k)}^m\,dt
\\
&\leq
\tau^{-1}C_1^{-m}\sum_{k=0}^{k_0-1} \int_{t_k}^{t_{k+1}} \norm{x(t)}^m\,dt
\\
&\leq
\tau^{-1}C_1^{-m} \int_{0}^{\infty} \norm{x(t)}^m\,dt. 
\end{aligned}
\end{equation}
Therefore 
\begin{equation}\label{eq:tausta2}
\sum_{k=0}^{\infty} \norm{x_d(k)}^m
\leq
(C_2^m+1)\tau^{-1}C_1^{-m} \int_{0}^{\infty} \norm{x(t)}^m\,dt. 
\end{equation}

Next consider the case when $h_k \geq \tau$ for every $k$. In a similar
way as \eqref{eq:tausta1} we can show that
\begin{equation*}
\begin{aligned}
\sum_{k=0}^\infty\norm{x_d(k)}^m
&\leq
\tau^{-1}\sum_{k=0}^\infty \int_{t_k}^{t_{k+1}} \norm{x_d(k)}^m\,dt
\\
&\leq
\tau^{-1}C_1^{-m} \sum_{k=0}^\infty \int_{t_k}^{t_{k+1}} \norm{x(t)}^m\,dt
\\
&=
\tau^{-1}C_1^{-m} \int_{0}^{\infty} \norm{x(t)}^m\,dt. 
\end{aligned}
\end{equation*} 

This inequality and \eqref{eq:tausta2} imply that there exists $C>0$
such that,  for every sample path $x$ of
{$\Sigma$}, it holds that
\begin{equation*}
\sum_{k=0}^\infty \norm{x_d(k)}^m 
\leq 
C \int_0^\infty \norm{x(t)}^m \,dt. 
\end{equation*}
Since {$\Sigma$} is stochastically $m$-th
mean stable, taking expectations in this inequality and using Fubini's
theorem show that $\mathcal S \Sigma^{(\tau)}$ is stochastically $m$-th
mean stable. 
\end{proof}

The next proposition shows the continuity of the matrix~$\mathcal
A_m(\Sigma^{(\tau)})$ at $\tau = 0$.
 
\begin{proposition}\label{proposition:tauLim}
It holds that 
\begin{equation}\label{eq:tauLim}
\lim_{\tau\to 0} \mathcal A_m(\Sigma^{(\tau)}) 
= 
\mathcal A_m(\Sigma). 
\end{equation}
\end{proposition}

\begin{proof}  
By \eqref{eq:hk>0} the function~$\chi_{\{h_k(\omega)\geq \tau\}}$
on~$\Omega$ increasingly converges to the constant~$1$ point-wise as
$\tau \to 0$ with probability 1. Hence, since each $A_i$ is Metzler, the
random variable $(J_k^{(\tau)} e^{A_{\sigma_k} h_k})^{[m]} =
\chi_{\{h_k(\omega)\geq \tau\}}(J_k e^{A_{\sigma_k} h_k})^{[m]}$
increasingly converges to $(J_k e^{A_{\sigma_k} h_k})^{[m]}$ as $\tau
\to 0$ point-wise with probability 1. Then the conditional monotone
convergence theorem (see, e.g., \cite{Borkar1995}) immediately shows
\eqref{eq:tauLim}.
\end{proof}

Before showing the proof of the necessity part  we need to introduce
another semi-Markovian jump system given by 
\begin{equation*}
\begin{aligned}
\Sigma_\alpha: 
\begin{cases}
\dfrac{dx_{\alpha}}{dt} 
= 
(A_{\sigma_k} +\alpha I) x_{\alpha}(t),\ t_k\leq t<t_{k+1}
\\
x_{\alpha}(t_{k+1}) = J_k x_{\alpha}(t_{k+1}^-),\ k\geq 0
\end{cases}
\end{aligned}
\end{equation*}
{for $\alpha\in\mathbb{R}$.} Notice that $\Sigma_0$ equals $\Sigma$.
About this system $\Sigma_\alpha$ we will need the next proposition.

\begin{proposition}\label{prop:sigmaalphasta}
If $\Sigma$ is exponentially $m$-th mean stable then there exists
$\alpha>0$ such that $\Sigma_\alpha$ is also exponentially $m$-th mean
stable.
\end{proposition}

\begin{proof}
Assume that there exist $C>0$ and $\beta>0$ such that
\eqref{eq:def:expsta} holds. Let $x_0$ and $\sigma_0$ be arbitrary.
Notice that, if $x(\omega ,\cdot)$ and $x_\alpha(\omega, \cdot)$ denote
the sample paths of $\Sigma$ and $\Sigma_\alpha$ for a fixed $\omega \in
\Omega$ with the common initial data $x_0$ and $\sigma_0$, respectively,
then we have $x_\alpha(\omega, t) = e^{\alpha t}x(\omega, t)$ because
$\Sigma$ and $\Sigma_\alpha$ share the same underlying stochastic
processes $\{\sigma_k\}_{k=0}^\infty$, $\{t_k\}_{k=0}^\infty$, and
$\{J_k\}_{k=0}^\infty$. Therefore we have $E[\norm{x_\alpha(t)}^m] =
e^{m\alpha t}E[\norm{x(t)}^m] \leq C e^{(m\alpha - \beta)t}\norm{x_0}^m$
for every $t\geq 0$. Hence $\Sigma_\alpha$ is exponentially $m$-th mean
stable if $0 < \alpha <\beta/m$. This completes the proof.
\end{proof}

{We will also use the next proposition.}

\begin{proposition}\label{proposition:incr}
$\rho(\mathcal A_m(\Sigma)) < \rho(\mathcal A_m(\Sigma_\alpha))$ for
every $\alpha > 0$. 
\end{proposition}
\begin{proof}
Let $i, j\in \{1, \dotsc, N \}$ and $r, s\in \{1, \dotsc, n_m \}$ be
arbitrary.  Assume that
\begin{equation*}
[[\mathcal A_m(\Sigma)]_{ij}]_{rs} 
=
\int_\Omega [(J_k e^{A_j h_k})^{[m]}]_{rs} \,dP_{ji}
>
0
\end{equation*}  
where $P_{ji}$ denotes the conditional probability distribution given by
$P_{ji}(\cdot) = P(\cdot \mid \sigma_{k}=i, \sigma_{k+1} = j)$.  By
Lemma~\ref{corollary:compSpecRads} it is sufficient to show that, for
these $i$, $j$, $r$, and~$s$, we have $[[\mathcal
A_m(\Sigma_\alpha)]_{ij}]_{rs} > [[\mathcal A_m(\Sigma)]_{ij}]_{rs}$.
Define $\Omega_\tau := \{\omega \in \Omega: h_k \geq \tau \}$. Since the
family of the measurable sets $\{ \Omega_\tau \}_{\tau > 0}$ covers
$\Omega$ minus a null set by \eqref{eq:hk>0}, there exists $\tau > 0$
such that {$ {\int_{\Omega_{\tau}} [(J_k e^{A_j h_k})^{[m]}]_{rs}
\,dP_{ji}} > 0.$} {Therefore}
\begin{equation*}
\begin{aligned}
{[[\mathcal A_m(\Sigma_\alpha)]_{ij}]_{rs}}
- [[\mathcal A_m(\Sigma)]_{ij}]_{rs}
&=
\int_{\Omega}
[(J_k e^{A_j h_k})^{[m]}]_{rs} 
(e^{m\alpha h_k} -1) \, dP_{ji}
\\
&\geq
\int_{\Omega_\tau}
[(J_k e^{A_j h_k})^{[m]}]_{rs} 
(e^{m\alpha h_k} -1) \, dP_{ji}
\\
&\geq
(e^{m\alpha \tau} -1) 
\int_{\Omega_\tau}
[(J_k e^{A_j h_k})^{[m]}]_{rs} 
\, dP_{ji}. 
\\
&{> 0, }
\end{aligned}
\end{equation*}
{as desired.}
\end{proof}

Now we are at the position to prove the necessity part of Theorem
\ref{theorem:main}. Notice that the two mappings $\Sigma \mapsto
\Sigma^{(\tau)}$ and $\Sigma\mapsto \Sigma_\alpha$ obviously commute so
that we can without confusion denote $(\Sigma^{(\tau)})_\alpha =
(\Sigma_\alpha)^{(\tau)}$ by $\Sigma_\alpha^{(\tau)}$. 

\defproof{the necessity part of Theorem~\ref{theorem:main}}
\begin{proof}
Assume that $\Sigma$ is exponentially $m$-th mean stable. Then, by
Proposition~\ref{prop:sigmaalphasta}, the system $\Sigma_\alpha$ is
exponentially $m$-th mean stable for some $\alpha>0$ as well. Since
$\Sigma_\alpha$ is clearly stochastically $m$-th mean stable, by
Proposition~\ref{proposition:tauSta} the discrete-time system~$\mathcal
S\Sigma_\alpha^{(\tau)}$ is stochastically $m$-th mean stable for every
$\tau>0$. Then Corollary~\ref{corollary:Schur<->DiscSta} shows
$\rho(\mathcal A_m(\Sigma_\alpha^{(\tau)})) < 1$. Hence, by
Proposition~\ref{proposition:incr} and
Proposition~\ref{proposition:tauLim},
\begin{equation*}
\begin{aligned}
\rho(\mathcal A_m(\Sigma))
&<
\rho(\mathcal A_m(\Sigma_\alpha))
\\
&=
\lim_{\tau \to 0}\rho(\mathcal A_m(\Sigma_\alpha^{(\tau)}))
\\
&\leq 1
\end{aligned}
\end{equation*}
because the spectral radius is a continuous function of a matrix. 
\end{proof}
\undefproof

Finally let us remark that, by proving Theorem~\ref{theorem:main}, we
have actually shown the next corollary.
\begin{corollary}
$\Sigma$ is exponentially $m$-th mean stable if and only if $\mathcal{S}
\Sigma$ is exponentially $m$-th mean stable.
\end{corollary}

\section{{Continuous-time Positive} Markovian Jump
Linear Systems} \label{section:mjls}

Though Theorem~\ref{theorem:main} gives the condition for the mean
stability of positive semi-Markovian jump linear systems, the condition
\ref{item:assm:h_k<T} in Assumption~\ref{assumption:cont} excludes
Markovian jump linear systems from the class of systems we can deal
with. The aim of this section is to give a stability characterization of
continuous-time Markovian jump linear systems.

Let us recall the definition of Markovian jump linear systems (see,
e.g., \cite{Feng1992}).  Let $\{r(t)\}_{t\geq 0}$ be a continuous-time
homogeneous Markov process taking values in the set $\{1, \dotsc, N\}$. 
The transition probabilities of the process $r(t)$ is given by, for
every $h>0$,
\begin{equation*}
P(r(t+h) = j \mid r(t) = i) = 
\begin{cases}
q_{ij} h + o(h),\ \ i\neq j,\\
1+q_{ii} h + o(h),\ \ i= j
\end{cases}
\end{equation*}
where $q_{ij}$ ($1\leq i,j\leq N$) are constants such that $q_{ij}\geq
0$ if $i\neq j$ and  $q_{ii} = -\sum_{j:j\neq i} q_{ij}$. The matrix $Q
= [q_{ij}]_{ij}$ is called \change{the transition matrix}{the
infinitesimal generator of the process~$\{r(t)\}_{t\ge 0}$}. Let $A_1$,
$\dotsc$, $A_N$ be $n\times n$ real matrices. Then the differential
equation
\begin{equation}\label{eq:MJLS}
\Sigma: 
\dfrac{dx}{dt} = A_{r(t)}x(t)
\end{equation}
is said to be a \emph{Markovian jump linear system}. {We assume that
$x(0) = x_0 \in \mathbb{R}^n$ and $r(0) = r_0 \in \{1, \dotsc, N\}$ are
constants.} {Notice that the} semi-Markovian jump linear system
\eqref{eq:SMJLS} is a Markovian jump linear system if $J_k = I$ and
\begin{equation*}
P(\sigma_{k+1}=j, h_k \leq t \mid \sigma_k = i) 
=
\begin{cases}
\dfrac{q_{ij}}{\sum_{k\neq i}q_{ik}}(1 - e^{-q_i t}), 
\quad (i\neq j)
\\
0. 
\quad (i=j)
\end{cases}
\end{equation*}
Finally we say that $\Sigma$ is {\it positive} if the matrices $A_1$,
$\dotsc$, $A_N$ are Metzler.

Since the condition~\ref{item:assm:h_k<T} \add{of
Assumption~\ref{assumption:cont}} is not satisfied we cannot use
Theorem~\ref{theorem:main} to check the stability of the Markovian jump
linear system~\eqref{eq:MJLS}. The aim of this section is to give the
next alternative characterization of the mean stability of positive
Markovian jump linear systems.
\begin{theorem}\label{theorem:MJLSstbl}
$\Sigma$ is exponentially $m$-th mean stable if and only if the matrix
\begin{equation} \label{eq:defn:T1}
\mathcal T_m := Q^\top \otimes I_{n_m} + \diag((A_1)_{[m]}, \dotsc, (A_N)_{[m]}) \in \mathbb{R}^{(n_mN \times  n_m N)}
\end{equation}
is Hurwitz stable. 
\end{theorem}

\begin{remark}
This theorem extends the well-known characterization~\cite{Feng1992} of
the mean square stability of Markovian jump linear systems to positive
Markovian jump linear systems. 
\end{remark}

The proof follows a similar way as that of Theorem~\ref{theorem:d:stbl}.
\change{First we}{We} introduce the stochastic process~$\{\delta(t)
\}_{t\geq 0}$ that takes its value in the set of standard unit vectors
in~$\mathbb{R}^N$ and is defined by
\begin{equation*}
\delta(t) = e_{r(t)}. 
\end{equation*}
Then $x$ satisfies the differential equation
\begin{equation}\label{eq:dx}
dx = \left(\sum_{i=1}^N \delta_i A_i \right)x\,dt, 
\end{equation}
{which is a generalization of the system studied in \cite{Hanlon2011a}.}
Applying the operator $(\cdot)^{[m]}$ to this equation yields, by
\eqref{eq:def:A_[m]}, 
\begin{equation}\label{eq:dx^[m]}
dx^{[m]} = \left(\sum_{i=1}^N \delta_i (A_i)_{[m]} \right)x^{[m]}\,dt. 
\end{equation}
\add{We can derive a differential equation also for $\delta$ as
follows.} \change{If}{Let} $\Pi_{ij}$ denote\del{s} the Poisson process
of the rate~$q_{ij}$ for every distinct pair~$(i, j)$\add{.} \add{Also
define}\label{pg:defn:Eij}
\begin{equation*}
\add{E_{ij} := e_i e_j^\top.}
\end{equation*}
\change{one can see}{Then it can be seen} that~\cite{Brockett2009}
\begin{equation}\label{eq:ddelta} 
d\delta 
= 
\sum_{i,j\add{: i\neq j}} 
(E_{ji}-E_{ii})\delta \ d\Pi_{ij}.
\end{equation}

\change{The next proposition}{Using the differential
equations~\eqref{eq:dx} and \eqref{eq:ddelta} we can prove the next
proposition, which} is similar to
Proposition~{\ref{proposition:diffeq}}.

\begin{proposition} \label{prop:inducedconf}
It holds that 
\begin{equation}\label{eq:key1}
\frac{d}{dt}E[\delta \otimes x^{[m]}] 
= 
\mathcal T_m E[\delta \otimes x^{[m]}]. 
\end{equation}
\end{proposition}

\begin{proof} 
First let us consider the case $m=1$. The It\^o rule for jump
processes~\cite{Brockett2009} applied to the variable~$\delta \otimes x$
shows that\add{, by \eqref{eq:dx} and \eqref{eq:ddelta}, }
\begin{equation}\label{eq:proof:diffeq:z_ox_x}
\frac{d}{dt}E[\delta\otimes x] 
= 
E\left[\delta \otimes \sum_{i=1}^N \delta_i A_i x\right]
+ 
E\left[\sum_{i,j\add{: i\neq j}} \left[(\mathchange{\Pi_{ji}}{{E}_{ji}}-\mathchange{\Pi_{ii}}{{E}_{ii}})\delta\add{q_{ij}}\right]\otimes x\right] \stmath{q_{ij}}. 
\end{equation}
Since $\delta_i\delta_j = 0$ if $i\neq j$ and $\delta_i^2 = \delta_i$
for every $i$, from the definition of Kronecker products it follows that
\begin{equation*}
\begin{aligned}
\delta\otimes \sum_{i=1}^N \delta_i A_i x
&=
\begin{bmatrix}
\delta_1 A_1 x
\\
\vdots
\\
\delta_N A_N x
\end{bmatrix}
\\
&=
\diag\left(A_1, \dotsc, A_N\right) (\delta\otimes x). 
\end{aligned}
\end{equation*}
Moreover, since $[(\mathchange{\Pi_{ji}}{{E}_{ji}}-\mathchange{\Pi_{ii}}{{E}_{ii}})
\delta \add{q_{ij}}]) \otimes x =
(\add{q_{ij}}(\mathchange{\Pi_{ji}}{{E}_{ji}}-\mathchange{\Pi_{ii}}{{E}_{ii}})\otimes
I_{n})(\delta \otimes x)$ and
\begin{equation*}
\begin{aligned}
\sum_{i,j: i\neq j}q_{ij} (\mathchange{\Pi_{ji}}{{E}_{ji}}-\mathchange{\Pi_{ii}}{{E}_{ii}})
&=
\sum_{i,j: i\neq j}q_{ij} \mathchange{\Pi_{ji}}{{E}_{ji}} + \sum_{i=1}^N \sum_{j\neq i} (-q_{ij}) \mathchange{\Pi_{ii}}{{E}_{ii}}
\\
&=
\sum_{i,j: i\neq j}q_{ij} \mathchange{\Pi_{ji}}{{E}_{ji}} + \sum_{i=1}^N q_{ii} \mathchange{\Pi_{ii}}{{E}_{ii}}
\\
&=
Q^\top, 
\end{aligned}
\end{equation*}
\change{the right hand side of the second term}{the second term of the right hand side} in \eqref{eq:proof:diffeq:z_ox_x}
equals $(Q^\top \otimes I_{n})E[\delta\otimes x]$. This shows
\eqref{eq:key1} for $m=1$. For a general $m$ notice that for $y :=
x^{[m]}$ we have $dy = \left(\sum_{i=1}^N \delta_i (A_i)_{[m]}
\right)y\,dt$ \add{by \eqref{eq:dx^[m]}}. Then we can proceed in the
same way as above to prove $\frac{d}{dt} E[\delta \otimes y] = \mathcal
T_m E[\delta \otimes y]$, which is \eqref{eq:key1}.
\end{proof}

Let us prove Theorem~\ref{theorem:MJLSstbl}. 
\defproof{Theorem~\ref{theorem:MJLSstbl}}
\begin{proof}
Notice that, by \eqref{eq:key1}, 
\begin{equation}\label{eq:d:diffcsqc}
E[\delta(t) \otimes x(t)^{[m]}] = e^{\mathcal T_m t} (\delta_0 \otimes x_0^{[m]}). 
\end{equation}

First assume that $\mathcal T_m$ is Hurwitz stable. Then there exist
$C>0$ and $\beta>0$ such that, for every $y \in \mathbb{R}^{n_m N}$ and
$t\geq 0$,
\begin{equation}\label{eq:expconv}
\norm{e^{\mathcal T_m t} y}_1 \leq C e^{-\beta t}
\twonorm{y}. 
\end{equation}
Let us show that  $\Sigma$ is exponentially $m$-th mean stable. Let
$x_0$ and $r_0$ be arbitrary. As in Lemma~\ref{lemma:x0>0wlog} we can
assume $x_0\geq 0$ without loss of generality. Then\add{, by the
equivalence of the norms on $\mathbb{R}^{n_m}$, there exists $C_1 > 0$
such that}
\begin{equation*}
\begin{aligned}
E[\twonorm{x(t)}^m]
&=
E[\twonorm{x(t)^{[m]}}]
\\
&\leq
C_1 E[\onenorm{x(t)^{[m]}}]
\\
&=
C_1 E[\onenorm{\delta(t) \otimes x(t)^{[m]}}]
\\
&=
C_1 \onenorm{ E[\delta(t) \otimes x(t)^{[m]}]}
\end{aligned}
\end{equation*}
by the linearity of $\onenorm{\cdot}$ on a positive orthant. This
inequality shows the exponential $m$-th mean stability of~$\Sigma$
because \eqref{eq:d:diffcsqc} and \eqref{eq:expconv} gives
\begin{equation*}
\begin{aligned}
\onenorm{ E[\delta(t) \otimes x(t)^{[m]}] } 
&\leq 
C e^{-\beta t} \twonorm{\delta_0 \otimes x_0^{[m]}} 
\\
&= 
C e^{-\beta t} \twonorm{x_0}^m. 
\end{aligned}
\end{equation*}

On the other hand assume that $\Sigma$ is exponentially $m$-th mean
stable. Let $C>0$ and $\beta>0$ be constants satisfying
\eqref{eq:def:expsta}.  By \eqref{eq:d:diffcsqc},
\begin{equation*}
\begin{aligned}
\twonorm{e^{\mathcal T_m t} (\delta_0 \otimes
x_0^{[m]}) }
&\leq
E[\twonorm{\delta(t) \otimes x(t)^{[m]}}]
\\
&=
E[\twonorm{x(t)}^m]
\\
&\leq
Ce^{-\beta t}\twonorm{x_0}^m
\end{aligned}
\end{equation*}
for every $x_0\in \mathbb{R}^n$ and $\delta_0$. Therefore, by Lemma
\ref{lemma:cone}, $e^{\mathcal T_m t} y$ converges to $0$ for every $y
\in \mathbb{R}^{n_mN}$, which proves that $\mathcal T_m$ is Hurwitz
stable.
\end{proof}
\undefproof

\subsection{{Output Feedback Stabilization}} 

As an illustration of Theorem~\ref{theorem:MJLSstbl} this section
studies the stabilization of positive Markovian jump linear systems in
the first mean, following the setting in \cite{Feng2010} where the
authors study the stabilization of Markovian jump linear systems in the
mean square. Consider the Markovian jump linear system with input and
output defined by
\begin{equation}\label{eq:MJLSio}
\begin{aligned}
\dfrac{dx}{dt} &= A_{r(t)}x(t) + B_{r(t)}u(t)\\
y &= C_{r(t)}x(t), 
\end{aligned}
\end{equation}
where $x(t) \in \mathbb{R}^n$, $u(t) \in \mathbb{R}^{n_u}$, $y(t) \in
\mathbb{R}^{n_y}$, and $A_i$, $B_i$, and $C_i$ ($i=1$, $\dotsc$, $N$)
are real matrices with appropriate dimensions. We study the
stabilization of the system \eqref{eq:MJLSio} via the output feedback
\begin{equation*}
u(t) = K_{r(t)}y(t),
\end{equation*} 
where $K_1$, $\dotsc$, $K_N \in \mathbb{R}^{n_u \times n_y}$. Then
the controlled system is described by
\begin{equation*}{
\Sigma_K:  \dfrac{dx}{dt} = A_{K,r(t)} x(t),\ 
A_{K,i} = A_i+ B_{i}K_{i}C_{i}
}\end{equation*}
which is again a Markovian jump linear system.

We consider the following problem. Let $\mathcal{T}_1(\Sigma_K)$
denote the matrix $\mathcal T_1$ given by \eqref{eq:defn:T1} for the
system $\Sigma_K$.

\begin{problem}\label{prb:MJLSstab}\upshape
Assume that the matrices $A_1$, $\dotsc$, $A_N$ are Metzler. Find
feedback gain matrices $K_1$, $\dotsc$, $K_N$ and a transition matrix
$Q$ such that the controlled system~$\Sigma_K$ is positive and
$\eta(\mathcal T_1(K))$ is minimized.
\end{problem}

This problem aims to stabilize a given positive Markovian jump linear
system in the first mean while keeping its positivity. Notice that not
only the gain matrices but also the transition matrix can be designed.
We can reduce the problem to an optimization without constraints as
follows. For $A \in \mathbb{R}^{n\times n}$ define its distance from the
set of $n\times n$ Metzler matrices $\mathbb{M}_{n}$ by $d(A,
\mathbb{M}_n) := \sum_{i=1}^n \sum_{j\neq i} \max(-A_{ij}, 0)$. Clearly
$A$ is Metzler if and only if $d(A, \mathbb{M}_n) = 0$. Then define
\begin{equation}\label{eq:functional}
f(K_1, \dotsc, K_N, Q) 
:= 
\eta(\mathcal{T}_1(\Sigma_K)) + 
\Gamma \left(\sum_{i=1}^N d(A_{K,i}, \mathbb{M}_n)
+
\sum_{i\neq j} \max(-q_{ij}, 0)\right)
\end{equation}
where $\Gamma >0$ is a constant. When $\Gamma$ is sufficiently large,
Problem~\ref{prb:MJLSstab} is almost equivalent to the minimization of
$f$ with respect to the free parameters~$K_1$, $\dotsc$, $K_N \in
\mathbb{R}^{n_u \times n_y}$ and $q_{ij}\in \mathbb{R}$ ($i\neq j$)
because the second term of $f$ forces $\Sigma_K$ to be positive and $Q$
be an admissible transition matrix. One can solve this minimization
problem using the gradient sampling algorithm proposed
in~\cite{Burke2005}, which has been applied to the design of low-order
and fixed-order controllers successfully~\cite{Burke2005,Burke2006}.

\begin{example}\upshape
Consider the Markovian jump linear system \eqref{eq:MJLSio} with
the two subsystems given by the triples
\begin{equation*}{
\begin{gathered}
(A_1, B_1, C_1) = \left(
\begin{bmatrix}0&0.2\\0.9&0.9\end{bmatrix}, 
\begin{bmatrix}0.6\\0.3\end{bmatrix},
\begin{bmatrix}0.3&0.1\end{bmatrix}
\right), 
\\
(A_2, B_2, C_2) = \left(
\begin{bmatrix}0.1&0.4\\0.6&-0.3\end{bmatrix}, 
\begin{bmatrix}0.2\\0.8\end{bmatrix}. 
\begin{bmatrix}-0.8&1\end{bmatrix}
\right). 
\end{gathered}
}\end{equation*}
We set $x_0 = [1\quad 1]^\top$ and $r_0 = 1$. We minimize the
function~\eqref{eq:functional} using a MATLAB
implementation~\cite{Burke} of the gradient sampling
algorithm~\cite{Burke2005} with $\Gamma = 10^5$ and obtain
\begin{equation*}{
K_1 = {-3.3333},\ 
K_2={-2.0000},\ 
Q = \begin{bmatrix}-2068.3&2068.3\\3123.1&-3123.1\end{bmatrix}
}\end{equation*}
that make $\Sigma_K$ positive and achieve $\eta(\mathcal
T_1(\Sigma_K)) = -0.1936$. A sample path of $\Sigma_K$is shown in
Fig.~\ref{figure:faststabilization}.
\begin{figure}[tb]
\centering
\includegraphics[width=9cm]{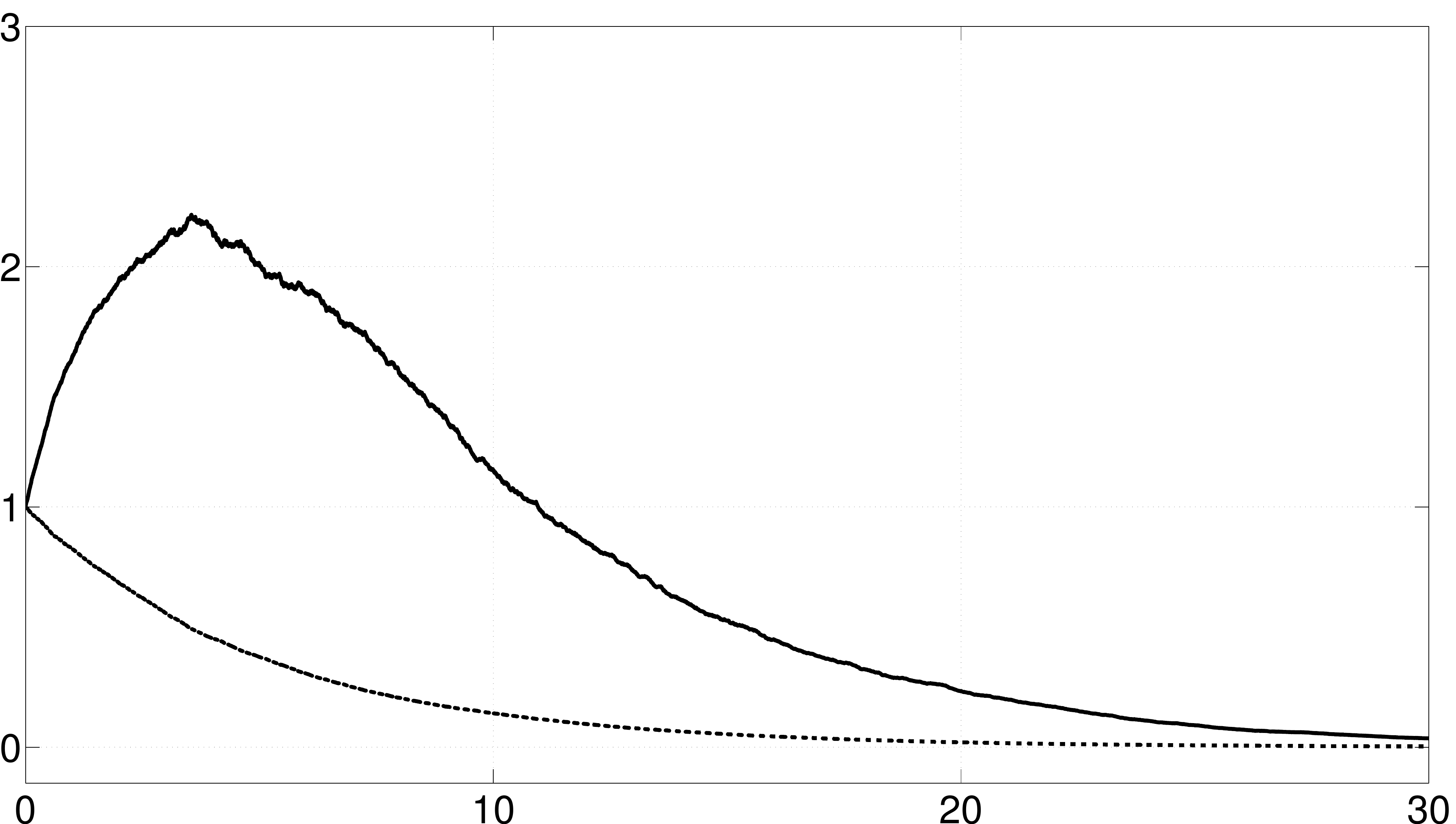}
\caption{{A sample path of the stabilized system. Dotted: $x_1(t)$. Solid: $x_2(t)$.}}
\label{figure:faststabilization}
\end{figure}

Though the above obtained parameters achieve stabilization in the first
mean, the transition matrix yields rather fast switchings that may not
be desirable in practice. To achieve stabilization with slower switching
rates let us consider the next modified objective function
\begin{equation*}
g(K_1, \dotsc, K_N, Q) 
= 
f(K_1, \dotsc, K_N, Q) 
+ 
\Gamma \sum_{i\neq j} \max({q_{ij}} - \bar q, 0)
\end{equation*}
where $f$ is defined in \eqref{eq:functional} and $\bar q > 0$ is a
constant. The second term of~$g$ aims to confine each off-diagonal entry
of $Q$ less than $\bar q$. Minimizing this function with $\Gamma = 10^5$
and $\bar q = 2$ gives
\begin{equation*}
K_1 = -3.3308,\ 
K_2=-1.9998,\ 
Q = \begin{bmatrix}-1.9997&1.9997\\1.9817&-1.9817\end{bmatrix}. 
\end{equation*}
With these parameters $\Sigma_K$ is still positive and first mean stable
as $\eta(\mathcal T_1(\Sigma_K)) =  -0.02251$. A sample path of the
stabilized system is shown in Fig.~\ref{figure:stabilization}.
\begin{figure}[tb]
\centering 
\includegraphics[width=9cm]{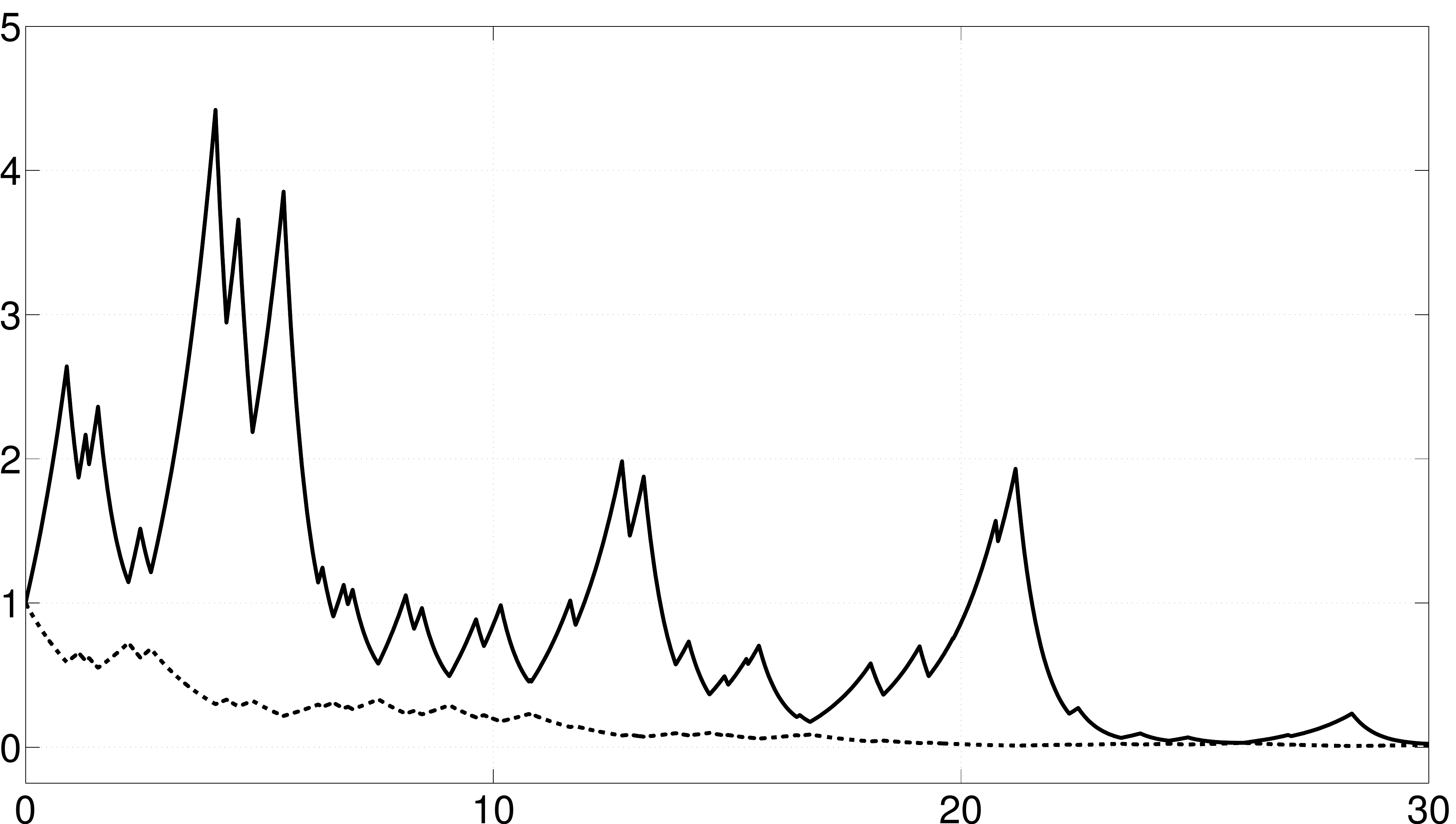}
\caption{{A sample path of the stabilized system with slower switching
rates. Dotted: $x_1(t)$. Solid: $x_2(t)$.}}
\label{figure:stabilization}
\end{figure}
\end{example}

\section{Conclusion} 

This paper studied the mean stability of positive semi-Markovian jump
linear systems. We showed that the mean stability is determined by the
spectral  radius of an associated matrix that is easy to compute. For
deriving the  condition we used a discretization of a semi-Markovian
jump linear system that preserves stability. Also we have given a
characterization for the exponential mean stability of continuous-time
positive Markovian jump linear systems. We illustrated the obtained
results with numerical examples.


\appendix

\section{Proof of Lemma~\ref{lemma:indp}}
\label{appendix:proof:lemma:indp}

Let $B_1 \subset \mathbb{R}^{n\times n}$ and $B_2 \subset
\mathbb{R}^{n}$ be arbitrary Borel sets. We need to show
\begin{equation}\label{eq:prf:tech}
P'\left(\zeta(k+1)_j F_k \in B_1, x_d(k)\in B_2\right) 
=
P'\left(\zeta(k+1)_j F_k \in B_1\right) 
P'\left(x_d(k) \in B_2\right). 
\end{equation}
The left hand side of this equation can be computed as
\begin{equation}\label{eq:prf:tech:1}
\begin{aligned}
&P'\left(\zeta(k+1)_j F_k \in B_1, x_d(k)\in B_2\right) 
\\
=&
\frac{1}{P(\Omega')}P\left(\zeta(k+1)_j F_k \in B_1, x_d(k)\in B_2, \sigma(k) = i\right)
\\
=&
P\left(\zeta(k+1)_j F_k \in B_1 \mid x_d(k)\in B_2, \sigma(k) = i\right)
P\left(x_d(k) \in B_2 \mid \sigma(k) = i\right)
\\
=&
P\left(\zeta(k+1)_j F_k \in B_1 \mid \sigma(k) = i\right)
P(x_d(k) \in B_2 \mid \sigma(k) = i)
\end{aligned}
\end{equation}
where in the last equation we used the condition~\ref{item:d:assm:renew}
of Assumption~\ref{definition:d}.  Then it is easy to show that
\begin{equation} 
\begin{aligned}
P(\zeta(k+1)_j F_k \in B_1 \mid \sigma(k) = i)
&=
\frac{1}{P(\Omega')}P(\zeta(k+1)_j F_k \in B_1 ,\sigma(k) = i)
\\
&=
P'(\zeta(k+1)_j F_k\in B_1). 
\end{aligned}
\end{equation}
In a similar way we can see that 
\begin{equation}\label{eq:prf:tech:3}
P(x_d(k) \in B_2 \mid \sigma(k) = i)
=
P'(x_d(k)\in B_2). 
\end{equation}
The equations \eqref{eq:prf:tech:1} to \eqref{eq:prf:tech:3} prove
\eqref{eq:prf:tech}. \endproof

\end{document}